\documentclass[twoside]{amsart}
\usepackage{amsmath}
\usepackage[norelsize]{algorithm2e}
\usepackage[psamsfonts]{amssymb}
\usepackage[all]{xy}
\usepackage{graphicx}
\usepackage[T1]{fontenc}
\usepackage[bitstream-charter]{mathdesign}
\usepackage{pifont}
\usepackage{bbm}
\DeclareSymbolFont{usualmathcal}{OMS}{cmsy}{m}{n}
\DeclareSymbolFontAlphabet{\mathcal}{usualmathcal}
\usepackage{hyperref}
\usepackage{tikz}
\newtheorem{theorem}{Theorem}[section]

\newtheorem{lemma}[theorem]{Lemma}

\newtheorem{claim}[theorem]{Claim}
\theoremstyle{remark}
\newtheorem{ex}[theorem]{Example}
\newtheorem{dfn}[theorem]{Definition}
\newtheorem{remark}[theorem]{Remark}
\def\A{\mathcal{A}}
\def\B{\mathcal{B}}
\def\P{\mathcal{P}}

\def\R{\mathbb{R}}
\def\Z{\mathbb{Z}}
\def\N{\mathbb{N}}

\def\D{\mathbb{D}}
\def\eps{\varepsilon}
\def\Q{\mathbb{Q}}
\def\tp{t_+}
\def\tmi{t_-}
\def\tm1{t_{-1}}
\def\t1{t_1}
\usepackage{subfloat}
\usetikzlibrary{arrows}
\newcommand{\StatBars}[6][]
{		\tikz;
		\draw[thick] (#2, #3) -- (#2+#4, #3);
		\node[fill = #5, shape = circle, scale = .5] at (#2, #3) {};
		\node[fill = #6, shape = circle, scale = .5] at (#2+#4, #3) {};
		
}

\newcommand{\GrowBars}[6][]
{		\tikz;
		\draw[thick] (#2, #3) -- (#2+#4, #3);
		\node[fill = #5, shape = circle, scale = .5] at (#2, #3) {};
		\node[fill = #6, shape = circle, scale = .5] at (#2+#4, #3) {};
		\draw[-triangle 90] (#2+#4-0.45,#3)-- +(0.25,0);

}

\newcommand{\GrowBarsLeft}[6][]
{		\tikz;
		\draw[thick] (#2, #3) -- (#2 - #4, #3);
		\node[fill = #5, shape = circle, scale = .5] at (#2, #3) {};
		\node[fill = #6, shape = circle, scale = .5] at (#2 - #4, #3) {};
		\draw[-triangle 90] (#2 - #4 + 0.45,#3) -- (#2 - #4 + 0.45 - 0.25, #3);

}

\newcommand{\GrowBarsLeftLight}[6][]
{		\tikz;
		\draw[thick, black!50] (#2, #3) -- (#2 - #4, #3);
		\node[fill = #5!50, shape = circle, scale = .5] at (#2, #3) {};
		\node[fill = #6!50, shape = circle, scale = .5] at (#2 - #4, #3) {};
		\draw[-triangle 90, black!50] (#2 - #4 + 0.45,#3) -- (#2 - #4 + 0.45 - 0.25, #3);

}

\newcommand{\ParBars}[5][]
{		\tikz;
		\draw[thick] (#2, #3) -- (#2 #5 0.75, #3);
		\node[fill = #4, shape = circle, scale = .5] at (#2, #3) {};
		\draw[dotted] (#2 #5 0.75, #3) -- (#2 #5 0.75 #5 0.75, #3);
		
}

\newcommand{\VertDash}[3][]
{		\tikz;
		\draw[dashed] (#2, 0) -- (#2, 5.5);
		\draw[-] (#2, -0.5) -- (#2, 0);
		\node[below] at (#2, -0.5) {#3};
}

\newcommand{\ParBarsLight}[5][]
{		\tikz;
		\draw[thick, black!50] (#2, #3) -- (#2 #5 0.75, #3);
		\node[fill = #4!50, shape = circle, scale = .5] at (#2, #3) {};
		\draw[dotted, black!50] (#2 #5 0.75, #3) -- (#2 #5 0.75 #5 0.75, #3);
		
}

\newcommand{\GrowParBarsLight}[4][]
{		\tikz;
		\draw[thick, black!50] (#2 - 0.75, #3) -- (#2, #3);
		\draw[dotted, black!50] (#2 - 1.5, #3) -- (#2 - 0.75, #3);
		\draw[-triangle 90, black!50] (#2 - 0.45, #3) -- +(0.25,0);
		\node[fill = #4!50, shape = circle, scale = .5] at (#2, #3) {};
}

\begin{document}

\title{A Quasi-Isometric Embedding into the group of Hamiltonian Diffeomorphisms with Hofer's Metric}
\author{Bret Stevenson}

\address{Department of Mathematics, University of Georgia, Athens, GA 30602}

\email{bret@math.uga.edu}

\begin{abstract}
We construct an embedding $\Phi$ of $[0,1]^{\infty}$ into $Ham(M, \omega)$, the group of Hamiltonian diffeomorphisms of a suitable closed symplectic manifold $(M, \omega)$.  We then prove that $\Phi$ is in fact a quasi-isometry.  After imposing further assumptions on $(M, \omega)$, we adapt our methods to construct a similar embedding of $\R \oplus [0,1]^{\infty}$ into either $Ham(M, \omega)$ or $\widetilde{Ham}(M, \omega)$, the universal cover of $Ham(M, \omega)$.  Along the way, we prove results related to the filtered Floer chain complexes of radially symmetric Hamiltonians.  Our proofs rely heavily on a continuity result for barcodes (as presented in \cite{Uzi}) associated to filtered Floer homology viewed as a persistence module.
\end{abstract}

\maketitle

\tableofcontents

\begin{section}{Introduction}

Let $(M, \omega)$ be a $2n$-dimensional closed symplectic manifold.  A smooth function $H: \R / \Z \times M \rightarrow \R$ defines a time-dependent vector field $X_H(t, \cdot)$ on $M$ by
\[
\omega(X_H(t, \cdot), \cdot) = -d(H_t),
\]
where $H_t = H(t, \cdot).$  A Hamiltonian isotopy $\phi_H^t$ is defined by letting $\phi_H^t$ be the time-$t$ map of the flow of $X_H(t, \cdot)$, while a Hamiltonian diffeomorphism $\phi$ is the time-1 map $\phi_H^1$ of such an isotopy.  The collection of all Hamiltonian diffeomorphisms forms a group $Ham(M, \omega)$, and to every $\phi$ of $Ham(M, \omega)$ we may associate its \textit{Hofer norm}
\[
||\phi||_{H} = \text{inf} \left\{ \int_0^1 \left( \underset{M}{\text{max}}(H_t) - \underset{M}{\text{min}}(H_t) \right) \, dt \, | \, \phi_H^1 = \phi \right\}.
\]
\textit{Hofer's metric} $d_{H}$ on $Ham(M, \omega)$ is then defined as
\[
d_{H}(\phi, \psi) = ||\phi^{-1} \circ \psi||_{H}
\]
for any $\phi, \psi \in Ham(M, \omega)$.

Now let $[0,1]^{\infty}$ denote the set of all $[0,1]$-valued sequences with only finitely many non-zero entries, and for $a = \{a_i\}_{i \geq 1}, b = \{b_i\}_{i \geq 1} \in [0,1]^{\infty}$, let 
\[
||a - b||_{\ell^{\infty}} = \text{max}_i|a_i - b_i|.
\]
With these notations established, we may state our main theorem as follows.

\begin{theorem}\label{main-theorem}  Let $M$ be a closed symplectic manifold which is either monotone or negative monotone.  Suppose we may symplectically embed a ball $B(2\pi R)$ of radius $\sqrt{2R}$ into $M$, where if $M$'s rationality constant $\gamma$ is non-zero, we require $4 \pi R \leq \gamma$.   Then for any $\eps > 0$, there exists an embedding $\Phi : [0,1]^{\infty} \rightarrow Ham(M, \omega)$ satisfying
\[
2 \pi R||a - b||_{\ell^{\infty}} - \eps \leq d_{H}(\Phi(a), \Phi(b)) \leq 4 \pi R ||a - b||_{\ell^{\infty}}
\]
for any $a, b \in [0,1]^{\infty}$.  That is, $\Phi$ is a quasi-isometric embedding of $[0,1]^{\infty}$ into $\text{Ham}(M, \omega)$.

\end{theorem}

Recall that a symplectic manifold is \textit{monotone} if $[\omega] |_{\pi_2(M)}  = \lambda c_1(TM)|_{\pi_2(M)}$ with $\lambda \geq 0$ and \textit{negative monotone} if the same relation holds but with $\lambda < 0$. Since the image of $c_1(TM)|_{\pi_2(M)}$ is a subgroup of $\Z$, the image of $[\omega]|_{\pi_2(M)}$ forms a discrete subgroup of $\R$ when $M$ is (negative) monotone; the \textit{minimal Chern number} $N$ is the non-negative generator of the image of $c_1(TM)|_{\pi_2(M)}$, and the \textit{rationality constant} $\gamma$ of $M$ is the non-negative generator of the image of $[\omega]|_{\pi_2(M)}$.  In particular, this paper has $N = 0$, $\gamma = 0$ when $c_1(TM)|_{\pi_2(M)} = 0$, $[\omega]|_{\pi_2(M)} = 0$, respectively; while this breaks from the convention of setting $N$, $\gamma$ equal to $\infty$ in such cases, we find it to be a worthwhile one for this work as it simplifies the discussion of  the various cases considered in the proof of Theorem \ref{main-theorem} (particularly Lemma \ref{big-lemma}).

Upon the introduction of Hofer's metric, it was natural to ask which symplectic manifolds $(M, \omega)$ yield $(Ham(M, \omega), d_H)$ with infinite diameter, and the appearances of results in this direction form a rich history; see \cite{LM}, \cite{Polt98}, \cite{Sch}, \cite{Ost}, \cite{EP}, and \cite{Mc}, for example.  Similarly, one may instead ask the broader question of which $Ham(M, \omega)$ admit quasi-isometric embeddings of multi-dimensional normed vector spaces.  This question already has partial answers, among which are results appearing in \cite{Py} and \cite{MikeSolo}.  Provided the existence of a closed Lagrangian $L \subset M$ which admits a Riemannian metric of non-positive curvature and has the inclusion-induced map $i_*: \pi_1(L) \rightarrow \pi_1(M)$ injective, Py shows that for any $m \in \N$ there exists an embedding $\phi: \Z^m \rightarrow Ham(M, \omega)$ and a constant $C_m > 0$ satisfying
\[
C_m^{-1}||a-b||_{\ell^{\infty}} \leq d_H(\phi(a), \phi(b)) \leq C_m||a-b||_{\ell^{\infty}}
\]
for any $a, b \in \Z^m$.  This result was generalized in \cite{MikeSolo}, in which Usher proves that if $M$ admits an autonomous Hamiltonian $H: M \rightarrow \R$ whose flow has all of its contractible periodic orbits constant, then there exists an embedding of $\R^{\infty}$ into $Ham(M, \omega)$ similar to the one presented in \cite{Py}.  It should be noted that Py's assumptions imply the existence of such an $H$, as explained in \cite{MikeSolo}.

While the conclusion of Theorem \ref{main-theorem} is much weaker than Usher's result and somewhat weaker than that of Py, its assumptions are quite mild and indeed do not lie entirely within the scope of these previous results.  For instance, Usher points out in \cite{MikeSolo} that any closed toric manifold $M$ will not admit an autonomous Hamiltonian $H$ as described in the previous paragraph, and so any such manifold which is also (negative) monotone (for example, $(S^2, \omega)$) is one for which Theorem \ref{main-theorem} asserts something new about the geometry of $(Ham(M, \omega), d_H)$.

The Hamiltonian diffeomorphisms which define our embedding are generated by radially symmetric functions $\bar{F}_i$ which are zero outside of $B(2\pi R)$ and of the form $\bar{f}_i \left( \tfrac{|z|^2}{2}\right)$, with $\bar{f}_i : [0, R] \rightarrow \R$, for $z \in B(2\pi R)$.  Each of our functions $\bar{f}_i$ are to have disjoint supports, each contained in $[R - \eps, R]$, so that the induced functions $\bar{F}_i$ have supports contained in the thin $2n$-dimensional annulus $\{ z \in B(2 \pi R) \, | \, 2R - 2\eps < |z|^2 < 2R\}$ near the boundary of $B(2\pi R)$.  (We note that using $\eps$ in this manner to construct our functions yields the inequality from Theorem \ref{main-theorem} with $\eps$ replaced by an appropriate scalar multiple.)  See Figure \ref*{sample_figures} for a piecewise linear version of one of our $\bar{f}_i$.  For an earlier application of such functions to questions of Hamiltonian dynamics, one may refer to \cite{Sey}, where Seyfaddini uses them to construct ``spectral killers.''  In fact, this paper employs several of the same strategies as \cite{Sey}, from the careful choices of perturbations of continuous, radially symmetric Hamiltonians, to the explicit enumerations of their actions.

\begin{figure}\label{sample_figures}

\begin{tikzpicture}[scale=.4]
		\draw[dashed] (-16, 4) -- (-3, 4);
		\draw[-] (-16,0) -- (-11,0);
		\draw[very thick] (-16, 0) -- (-11, 0);
		\draw[dotted] (-11, 0) -- (-7.5, 0);
		\draw[->] (-7.5,0) -- (18,0) node[right]{$r$};
		\draw[very thick] (-7.5, 0) -- (-3.95, 0) -- (-3,4) -- (-2.5, 0) -- (-2, 4) -- (-1.05, 0) -- (17, 0);	
		\draw[->] (-16, -0.5) -- (-16,6);
		\draw[-] (-7,-0.5) -- (-7, 0.5);
		\node[below] at (-7,-0.5) {$R - \eps$};
		\draw[-] (17,-0.5) -- (17, 0.5);
		\node[below] at (17,-0.5) {$R$};
		\node[below] at (-16, -0.5) {$0$};
		\draw[-] (-16.5, 4) -- (-15.5, 4);
		\node[left] at (-16.5, 4) {$2 \pi R$};
\end{tikzpicture}

\hspace*{0.80cm}
\begin{tikzpicture}[scale=.4]
		%\node[left] at (-16.5, 4) {$ 2 \pi R$};
		\draw[-] (-16,0) -- (-11,0);
		\draw[dotted] (-11, 0) -- (-7.5, 0);
		\draw[->] (-7.5,0) -- (18,0) node[right]{$r$};
		\draw[very thick] (-16, 0) -- (-11, 0);
		\draw[very thick] (-7.5, 0) -- (-5.45,0) -- (-5, 3.5) -- (-4.75, 0) -- (-4.5, 3.5) -- (-4.05, 0) -- (-3.95, 0) -- (-3,4) -- (-2.5, 0) -- (-2, 4) -- (-1.05, 0) -- (5.05, 0) -- (9, 1) -- (11, 0) -- (13, 1) -- (16.95, 0);	
		\draw[->] (-16, -0.5) -- (-16,6);
		\draw[-] (-7,-0.5) -- (-7, 0.5);
		\node[below] at (-7,-0.5) {$R - \eps$};
		\draw[-] (17,-0.5) -- (17, 0.5);
		\node[below] at (17,-0.5) {$R$};
		\node[below] at (-16, -0.5) {$0$};
\end{tikzpicture}

%\begin{tikzpicture}[scale=.4]
		%\draw[-] (-16,0) -- (-11,0);
		%\draw[dotted] (-11, 0) -- (-7.5, 0);
		%\draw[->] (-7.5,0) -- (18,0) node[right]{$r$};
		%\draw[very thick] (-16, 0) -- (-11, 0);
		%\draw[very thick] (-7.5, 0) -- (-6.95, 0) -- (-6.75, 2) -- (-6.625 ,0) -- (-6.5, 2) -- (-6.3, 0) -- (-5.45,0) -- (-5, 3.5) -- (-4.75, 0) -- (-4.5, 3.5) -- (-4.05, 0) -- (-3.95, 0) -- (-3,4) -- (-2.5, 0) -- (-2, 4) -- (-1.05, 0) -- (5.05, 0) -- (9, 1) -- (11, 0) -- (13, 1) -- (16.95, 0);	
		%\draw[->] (-16, -0.5) -- (-16,6);
		%\draw[-] (-7,-0.5) -- (-7, 0.5);
		%\node[below] at (-7,-0.5) {$R - \eps$};
		%\draw[-] (17,-0.5) -- (17, 0.5);
		%\node[below] at (17,-0.5) {$R$};
		%\node[below] at (-16, -0.5) {$0$};
%\end{tikzpicture}

\caption{The top figure is a piecewise linear version of one of our functions $\bar{f}_i$.  The bottom figure is a piecewise linear version of some $\sum_i^{\infty} a_i \bar{f}_i$, which will induce the Hamiltonian diffeomorphism $\Phi(a)$ with $a = \{ a_i \}_{i \geq 1} \in [0,1]^{\infty}$.}

\end{figure}
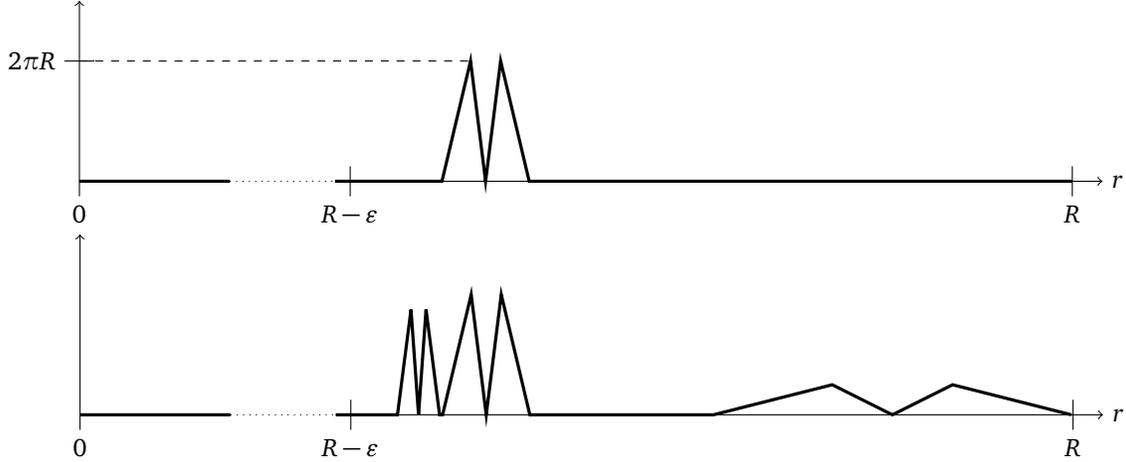

In an effort to build an analogous embedding of $\R \oplus [0,1]^{\infty}$ (where $\R \oplus [0,1]^{\infty}$ is defined similarly to $[0,1]^{\infty}$), we wish to find symplectic manifolds $(M, \omega)$ whose $Ham(M, \omega)$ admit a one-parameter family of diffeomorphisms $\phi_s$ satisfying $||\phi_s||_H \geq K \cdot s$ for some constant $K > 0$.  Such families can be shown to exist whenever there is a \textit{stable homogeneous Calabi quasi-morphism} $\mu: Ham(M, \omega) \rightarrow \R$; see \cite{EPZ} for a definition and details.  Using results from \cite{EP}, we have the following.

\begin{theorem}\label{quasimorphism-version}
Let $(M, \omega)$ and $B(2 \pi R)$ be as in the statement of Theorem \ref{main-theorem}, and further assume that

\begin{itemize}

\item  there exists a stable homogeneous Calabi quasi-morphism $\mu: Ham(M, \omega) \rightarrow \R$.

\item  $B(2 \pi R)$ is displaceable in $M$, i.e. there exists a Hamiltonian diffeomorphism $\phi: M \rightarrow M$ such that $\phi(B(2 \pi R)) \cap B(2 \pi R) = \emptyset$.

\end{itemize}

Then for any $\eps > 0$, there exits an embedding $\overline{\Phi}: \R \oplus [0,1]^{\infty} \rightarrow Ham(M, \omega)$ so that for any $a, b \in \R \oplus [0,1]^{\infty}$,
\[
C ||a - b||_{\ell^{\infty}} - \eps \leq d_{H}(\overline{\Phi}(a), \overline{\Phi}(b)) \leq 4 \pi R ||a - b||_{\ell^{\infty}},
\]
where 
\[
C = \left( \frac{2 \pi R \cdot \text{Vol}(B(2 \pi R))}{\text{Vol}(M)} - \eps \right).
\]
Here, $\text{Vol}(B(2 \pi R))$ and $\text{Vol}(M)$ are the symplectic volumes of $B(2 \pi R)$ and $M$, respectively.
\end{theorem}

In \cite{EP}, Entov and Polterovich explicitly construct a stable homogeneous Calabi quasi-morphism on $Ham(M, \omega)$ and outline sufficient conditions for which their construction holds.  The authors therein also elaborate on the existence of such quasi-morphisms for a few specific $(M, \omega)$.

\begin{ex}  Let $\eps > 0$, and consider $(S^2, \omega)$, the 2-sphere with the area form $\omega$ such that $\int_{S^2} \omega = 4 \pi$.  We may symplectically embed a displaceable disk of radius $\sqrt{2(1 - \eps)}$ into the Northern hemisphere, and \cite{EP} shows that $Ham(S^2, \omega)$ admits a stable homogeneous Calabi quasi-morphism.  Moreover, $(S^2, \omega)$ is monotone with rationality constant $4 \pi$.  We may therefore apply Theorem \ref{quasimorphism-version} to say that there exists an embedding $\overline{\Phi}: \R \oplus [0,1]^{\infty} \rightarrow Ham(S^2, \omega)$ satisfying
\[
(\tfrac{ 2 \pi (1 - \eps)}{4 \pi} - \eps)||a-b||_{\ell^{\infty}} - \eps \leq d_H(\overline{\Phi}(a), \overline{\Phi}(b)) \leq 4 \pi (1 - \eps) ||a-b||_{\ell^{\infty}}.
\]
\end{ex}

\begin{remark}\label{two_sphere_remark}
While it is again deduced in \cite{EP}, $Ham(S^2, \omega)$ having infinite diameter with respect to Hofer's metric dates back earlier to \cite{Polt98}.  However, it is still unknown whether a multi-dimensional normed vector space may be quasi-isometrically embedded into $Ham(S^2, \omega)$.  In fact, there is nothing as of yet which rules out the possibility of $Ham(S^2, \omega)$ lying inside an infinitely long tube of a fixed radius.  If this is the case, Theorem \ref{quasimorphism-version} and the example above therefore give a lower bound on what this radius can be.
\end{remark}

We may also consider embeddings of $\R \oplus [0,1]^{\infty}$ into $\widetilde{Ham}(M, \omega)$, the universal cover  of $Ham(M, \omega)$.  Elements of this universal cover are homotopy classes $\{ \phi_t \}$ of paths (rel. endpoints) of Hamiltonian diffeomorphisms.  Similar to the case of $Ham(M, \omega)$, we may define the \textit{Hofer pseudo-norm} $\widetilde{|| \cdot ||}_H$ by
\[
\widetilde{||\{ \phi_t \} ||}_H = \text{inf} \left\{ \int_0^1 \left( \underset{M}{\text{max}}(H_t) - \underset{M}{\text{min}}(H_t) \right) \, dt \, | \, H_t \, \text{generates the path} \, \{ \phi_t \} \right\},
\]
after which we may define the \textit{Hofer pseudo-metric} $\tilde{d}_H$ as in the case of $Ham(M, \omega)$.  Again based on results from \cite{EP} concerning stable homogeneous Calabi quasi-morphisms, as well as a result from \cite{Polt} about \textit{stably non-displaceable} Lagrangians, we have the following result.

\begin{theorem}\label{universal-analogue}

Let $(M, \omega)$ and $B(2 \pi R)$ be as in the statement of Theorem \ref{main-theorem}.  Further assume one of the following:

\begin{itemize}

\item  $M$ has a Lagrangian submanifold $L$ which is \textit{stably non-displaceable}, and $B(2 \pi R) \cap L = \emptyset$.

\item  there exists a stable homogeneous Calabi quasi-morphism $\tilde{\mu}: \widetilde{Ham}(M, \omega) \rightarrow \R$, and $B(2 \pi R)$ is displaceable in $M$.

\end{itemize}

Then for any $\eps > 0$, there exits an embedding $\widetilde{\Phi}: \R \oplus [0,1]^{\infty} \rightarrow \widetilde{Ham}(M, \omega)$ so that for any $a, b \in \R \oplus [0,1]^{\infty}$,
\[
C||a - b||_{\ell^{\infty}} - \eps \leq \tilde{d}_{H}(\widetilde{\Phi}(a), \widetilde{\Phi}(b)) \leq 4 \pi R ||a - b||_{\ell^{\infty}},
\]
where $C$ is as in Theorem \ref{quasimorphism-version}.
\end{theorem}

See \cite{EP}, \cite{EP2}, or \cite{UsherCalabi} for more information concerning stable homogeneous Calabi quasi-morphis-ms on $\widetilde{Ham}(M, \omega)$, as well as some examples of closed $(M, \omega)$ whose $\widetilde{Ham}(M, \omega)$ admit such a quasi-morphism; for instance, it is shown in \cite{UsherCalabi} that such $(M, \omega)$ include all closed toric manifolds, as well as any point blowup of an arbitrary closed symplectic manifold.  For the definition of ``stably non-displaceable,'' we refer the reader to \cite{EP3}.  For examples of stably non-displaceable Lagrangians, one may refer to \cite{EP3} or \cite{Polt}, where in the latter, a Lagrangian $L$ being stably non-displaceable is referred to as satisfying the \textit{stable Lagrangian intersection property}.

Our paper is organized as follows.  Section \ref*{HFH} recalls the basic construction of filtered Floer homology.  We then use Section \ref*{BC} to discuss persistence modules, barcodes, and their application to filtered Floer homology, including how the \textit{boundary depth} of a Hamiltonian diffeomorphism can be recovered from its barcode.  Section \ref*{RSH} reviews radially symmetric Hamiltonians, discusses how to associate barcodes to radially symmetric $C^0$ functions, and proves certain lemmas concerning these barcodes.  Section \ref*{Main_Proof} proves Theorem \ref{main-theorem}, while Section \ref*{Leftovers} proves Theorems \ref{quasimorphism-version} and \ref{universal-analogue}.

%Letting $(S^2, \omega)$ be the 2-sphere with the area form $\omega$ such that $\smallint_{S^2} \omega = 4 \pi$ gives a monotone manifold with rationality constant $4 \pi$, so the conclusion of Theorem \ref{main-theorem} holds for $Ham(S^2, \omega)$ with $R = 1$.  Moreover, where $\R \oplus [0,1]^{\infty}$ is defined similarly to $[0,1]^{\infty}$, the methods from our proof of Theorem \ref{main-theorem} may be used in conjunction with results from chapter 7 of \cite{Polt} to get the following.

%\begin{theorem}\label{sphere-theorem}  For any $\eps > 0$, there exits an embedding $\overline{\Phi}: \R \oplus [0,1]^{\infty} \rightarrow Ham(S^2, \omega)$ so that for any $a, b \in \R \oplus [0,1]^{\infty}$,
%\[
%(\pi - \eps) ||a - b||_{\ell^{\infty}} - \eps \leq d_{H}(\overline{\Phi}(a), \overline{\Phi}(b)) \leq 4 \pi ||a - b||_{\ell^{\infty}}.
%\]

%\end{theorem}

%Our paper is organized as follows.  Section 2 recalls the basic construction of filtered Floer homology.  We then use section 3 to discuss persistence modules, barcodes, and their application to filtered Floer homology, including how the \textit{boundary depth} of a Hamiltonian diffeomorphism can be recovered from its barcode.  Section 4 reviews radially symmetric Hamiltonians, discusses how to associate barcodes to radially symmetric $C^0$ functions, and proves certain lemmas concerning these barcodes.  Finally, sections 5 and 6 contain the proofs of Theorems \ref{main-theorem} and \ref{sphere-theorem}, respectively.

The author was partially funded by the BJ Ball Scholarship, which was awarded by the University of Georgia's Mathematics Department.  He is deeply grateful to his advisor, Michael Usher, for his patience and guidance throughout the progression of this work, as well as for helpful comments on earlier drafts of this paper.  He would also like to thank Sobhan Seyfaddini, for introducing him to some of the useful lemmas employed in \cite{SeyAnnulus}, and Jun Zhang, for useful discussions.  Finally, the author extends his gratitude to the referee for corrections and helpful comments.

\end{section}

\begin{section}{Hamiltonian Floer homology}\label{HFH}
Below, we recall the basic construction of the filtered Hamiltonian Floer homology $HF_*^{\tau}(H)$ associated to a non-degenerate Hamiltonian $H$ on a closed (negative) monotone manifold $M$.  For more details, we refer the reader to \cite{Floer} for the monotone case, \cite{Hof-Sal} for the semi-positive case, and \cite{Pardon} for the case of a general closed symplectic manifold.  For the remainder of this work, ``monotone'' will include the case of negative monotone.

For a smooth $H: \R / \Z \times M \rightarrow \R$, let $\phi_H^t$ be the induced Hamiltonian isotopy as defined in the introduction.  Let $x \in M$ be a fixed point of $\phi^1_H$ such that the 1-periodic orbit of $H$ given by $x(t) = \phi_H^t(x)$ is contractible in $M$.  We call $x(t)$ \textit{non-degenerate} if the time 1 map of the linearization of its flow has all eigenvalues not equal to 1 (i.e. $\textit{det}(\mathbbm{1} - d_{x(1)}(\phi^1_H)) \neq 0$), and we call $H$ non-degenerate if all contractible 1-periodic orbits of $H$ are non-degenerate.  A Hamiltonian $H$ being non-degenerate makes all of its fixed points isolated, so if $M$ is compact, the set $\P(H)$ of $H$'s contractible 1-periodic orbits must be finite.

Each $x(t) \in \P(H)$ can be capped by gluing a disk to $x(t)$ via a map $v: \D^2 \rightarrow M$ satisfying $v(e^{(2 \pi \sqrt{-1}) t}) = x(t)$.  We let either $[x(t), v]$ or $\bar{x}$ denote an equivalence class of capped $x(t)$, where two capped periodic orbits $[x(t), v]$ and $[y(t), w]$ are considered equivalent if $x(t) = y(t)$ and $c_1(TM)|_{\pi_2(M)} ([v \# \overline{w}])$ and  $\int_{S^2} (v \# \overline{w})^*\omega$ are both zero; here, $v \# \overline{w}$ is the sphere created by gluing $w$ to $v$ by an orientation-reversing map on their boundary. 

Given a capped periodic orbit $[x(t),v]$, we may symplectically trivialize $v^*(TM)$ and use this trivialization to express the linearization of $x(t)$'s flow as a path of symplectic matrices.  Assuming $x(t)$ is non-degenerate, the \textit{Conley-Zehnder} index $\mu_{CZ}([x(t),v])$ is an integer measuring the rotation of specific eigenvalues as we move through this path of matrices (see \cite{SZ}).  If $v$ and $w$ are two different cappings for $x(t)$, then $\mu_{CZ}([x(t), v]) - \mu_{CZ}([x(t), w]) = - 2c_1(TM)|_{\pi_2(M)}([v \# \overline{w}])$.  Different conventions are used in different works when defining the Conley-Zehnder index of a capped periodic orbit; our conventions are the same as those used in \cite{SZ} so that if $f$ is a $C^2$-small Morse function on the $2n$-dimensional $M$, a critical point of Morse index $j$ will have Conley-Zehnder index $j-n$ when treated as a trivially capped periodic orbit.

We note here that under our monotonicity condition, two capped periodic orbits $[x(t),v]$ and $[y(t),w]$ are equivalent if and only if $x(t) = y(t)$ and $\mu_{CZ}([x(t),v]) = \mu_{CZ}([y(t),w])$.  Indeed, we would have $c_1(TM)|_{\pi_2(M)}([v \# \overline{w}]) =0$ by the previous paragraph, so 
\[
\int_{S^2} (v \# \overline{w})^*\omega = \lambda c_1(TM)|_{\pi_2(M)}([v \# \overline{w}]) = 0
\]
as well.  Hence, for every periodic orbit $x(t) \in \P(H)$ and $d \in \Z$, there exists at most one equivalence class $[x(t), v]$ so that $\mu_{CZ} ([x(t), v]) = d$.  This and $\P(H)$ being finite implies that $\widetilde{\P}_d(H)$, the set of equivalence classes of capped periodic orbits of $H$ with Conley-Zehnder index $d$, is a finite set.  We may therefore construct a finite dimensional vector space over $\Q$ with generators the elements of $\widetilde{\P}_d(H)$, and we let $CF_d(H)$ denote this vector space.  This represents the $d$-th graded portion of $H$'s \textit{total Floer chain complex}, denoted by $CF_*(H)$.

\begin{remark}  We see that, generally, the total Floer chain complex is infinite dimensional over $\Q$.  One way of getting around this is by considering $CF_*(H)$ as a finite dimensional vector space over a \textit{Novikov ring} (see \cite{Hof-Sal}, for instance).  The previous paragraph shows why we have no need for a Novikov ring in our construction of the Floer chain complex, for we assume monotonicity and restrict our attention to each degree $d$-th portion.
\end{remark}

\begin{remark}\label{DefiningA}  Under our monotonicity assumption, every capping for a fixed periodic orbit $x(t)$ can be obtained by first fixing a capping $v$ and then attaching a multiple of an appropriate element of $\pi_2(M)$ to $v$.  To be precise, let $[A] \in \pi_2(M)$ and a capped periodic orbit $[x(t), v]$ be given.  Where $[x(t), v \# A]$ is the capped periodic orbit created by attaching the sphere $A$ to $v$, we have
\[
\mu_{CZ}([x(t), v \# A]) = \mu_{CZ}([x(t), v]) - 2c_1(TM)|_{\pi_2(M)}([A]).
\]
So choosing $[A]$ with $c_1(TM)|_{\pi_2(M)}([A]) = - N$, every possible capping of $x(t)$ is given by
\[
\{ [x(t), v \#  kA] \}_{k \in \Z},
\]
while the set of possible Conley-Zehnder indices is given by
\[
\{ \mu_{CZ}([x(t), v]) + 2Nk \}_{k \in \Z}.
\]
Here, $v \#  kA$ means $k$ copies of $A$ attached to $v$.  (Note that if $N = 0$, every capped periodic orbit in the first set is equivalent.)

\end{remark}

To describe the boundary operator $\partial_H$ of $CF_*(H)$, we first let $\widetilde{\mathcal{L}_0(M)}$ denote the space of all capped, contractible loops in $M$ endowed with the same equivalence relation used on capped periodic orbits.  For a given Hamiltonian $H$ on $M$, we can define the \textit{action functional} $\A_H$ on $\widetilde{\mathcal{L}_0(M)}$ by
\[
\A_H([\gamma(t), v]) = - \int_{\D^2} v^*(\omega) + \int_0^1 H(t, \gamma(t))dt,
\]
which is well-defined by our equivalence relation on capped periodic orbits.  The critical points of this action functional are precisely the capped periodic orbits of $H$, and when $H$ is non-degenerate, the boundary operator $\partial_H$ for $CF_*(H)$ is defined by a count of isolated (formal) negative gradient flowlines of $\A_H$ on $\widetilde{\mathcal{L}_0(M)}$.  (In the case that $M$ is semipositive, these may be more concretely defined, for generic choices of non-degenerate $H$ and time-dependent $\omega$-compatible almost-complex structure $J_t$, as isolated solutions $u: \R \times \R / \Z \rightarrow M$ to the Hamiltonian Floer equation
\[
\frac{\partial u}{\partial s} + J_t(u)\left( \frac{\partial u}{\partial t} - X_H(t,u)\right) = 0.
\]
If the capped periodic orbit $[y(t),w]$ has a non-zero coefficient in $\partial_H([x(t),v])$, then there exists such a $u$ which limits on $x(t)$ (resp. $y(t)$) as $s$ goes to negative (resp. positive) infinity and such that $[x(t), v] = [x(t), u \# w]$.  The resulting filtered homology, defined below, is independent of our choice of $J_t$).  It is true, though highly nontrivial to prove, that $\partial_H$ defined in this way gives well-defined maps $\partial_{H, \, d} : CF_d(H) \rightarrow CF_{d-1}(H)$ satisfying $\partial_{H, \, d-1} \circ \partial_{H, \, d} = 0$ for all degrees $d$.

After restricting $\A_H$ to $\cup_{d \in \Z}\widetilde{P}_d(H)$, we may extend it to a function $\ell$ on all of $CF_*(H)$ by setting
\[
\ell ( 0 ) = -\infty
\]
and
\[
\ell(c) = \underset{ \{ i \, | \, q_i \neq 0 \}}{\text{max}} (\A_H([x_i(t), v_i])
\]
for $c = \sum q_i [x_i(t), v_i]$ a non-zero chain in $CF_*(H)$.  It is known that $\ell(\partial(c)) < \ell(c)$ for such non-zero chains, so we may create the subcomplex $CF_*^{\tau}(H)$ of $CF_*(H)$ (where $\tau \in \R$) generated by capped periodic orbits with action less than or equal to $\tau$.  Letting $\partial_{H}^{\tau}$ denote the boundary operator of this subcomplex, we set $HF_*^{\tau}(H) = [\text{ker}(\partial_{H}^{\tau})] / [\text{Im} (\partial_{H}^{\tau})]$ to get the \textit{filtered Floer homology} of $H$; we write $HF_*(H)$ for $HF_*^{\infty}(H)$ and call it the \textit{total} Floer homology.

We take a final moment to recall that the \textit{action spectrum} $\textit{Spec}(H)$ of $H$ is simply the set $\cup_{d \in \Z} \A_H( \widetilde{\P}_d(H))$.  Later on, we may refer to a \textit{degree $d$ action of $H$}, by which we mean an element of $\A_H(\widetilde{P}_d(H))$.

\begin{remark}  It will be important to note here the effect of recappings on actions.  Where $[x(t), v]$ and $[A]$ are as from our previous remark, then 
\[
\A([x(t), v \# A]) = \A([x(t),v]) + [\omega]|_{\pi_2(M)}([A])
\]
which is equal to $\A([x(t),v]) + \sigma(\lambda) \gamma$ under our monotonicity condition; furthermore,
\[
\A([x(t), v \# kA]) = \A([x(t),v]) + k[\omega]|_{\pi_2(M)}([A]) = \A([x(t),v]) + k\sigma(\lambda) \gamma.
\]
Here, $\sigma(\lambda)$ is the sign of the monotonicity constant $\lambda$ (with $\sigma(0) = 0)$.  It is this fact that will allow us to enumerate all possible actions and degrees for the capped periodic orbits of certain non-degenerate Hamiltonians on monotone manifolds.

\end{remark}

\end{section}

\begin{section}{Barcodes}\label{BC}
For our discussion of persistence modules and barcodes, we mainly follow the expositions provided in \cite{Uzi} and \cite{PoltShek}.

\begin{subsection}{Persistence modules and barcodes}  Let $K$ be a field.  A \textit{persistence module} $\mathbb{V} = (V, \sigma)$ consists of a $K$-module $V_t$ for each $t \in \R$ and morphisms $\sigma_{st}: V_s \rightarrow V_t$, for each pair $s, t$ with $s \leq t$, such that $\sigma_{ss} = \text{Id}|_{V_s}$ and $\sigma_{tu} \circ \sigma_{st} = \sigma_{su}$.

For an easy example of a persistence module, we may construct an \textit{interval module} $\mathbb{M}(I) = (M(I), \sigma)$ by choosing an interval $I \subset \R$ and defining each $M(I)_t$ by

\begin{displaymath}
   M(I)_t = \left\{
     \begin{array}{lr}
       K, &  t \in I\\
       0, & \text{otherwise};
     \end{array}
   \right.
\end{displaymath}
our maps $\sigma_{st}: M(I)_s \rightarrow M(I)_t$ in this case will be the identity when $s, t \in I$ and the zero map otherwise.

As well as being an easy example of a persistence module, interval modules turn out to be the building blocks of other persistence modules satisfying certain conditions.  One such condition (as the following theorem asserts) is $\mathbb{V}$ being \textit{pointwise finite-dimensional}, where each $V_t$ is a finite-dimensional vector space.  (Another sufficient condition is $\mathbb{V}$ being of \textit{finite type} as in \cite{ZoomCar}.)

\begin{theorem} (\cite{BillCrawley})  Any pointwise finite-dimensional persistence module $\mathbb{V}$ can be uniquely expressed as a direct sum of interval modules $\mathbb{M}(I_{\alpha})$.

\end{theorem}

Thus, for a pointwise finite dimensional persistence module $\mathbb{V}$, we can define its \textit{barcode} as the collection $\B = \{(I_{\alpha}, m_{\alpha})\}$, where each $I_{\alpha}$ is an interval appearing in $\mathbb{V}$'s interval module decomposition with multiplicity $m_{\alpha} > 0$.  We may sometimes refer to an $I_{\alpha}$ with $(I_{\alpha}, m_{\alpha}) \in \B$ as a \textit{bar} or \textit{interval of $\B$}, while by a \textit{left} or \textit{right-hand endpoint of $\B$} we mean the left or right-hand endpoint of a bar of $\B$.

\begin{remark}  Let $H$ be non-degenerate on closed monotone $M$ and fix a degree $d$.  Referring to Section \ref*{HFH}, one sees that $CF_*^s(H)$ is a subcomplex of $CF_*^t(H)$ whenever $s \leq t$, and it is easily verified from here that we get a pointwise finite-dimensional persistence module by setting $V_t = HF_{d}^t(H)$ and $\sigma_{st}: HF_{d}^s(H) \rightarrow HF_{d}^t(H)$ equal to the map induced by inclusion on the chain level.  It therefore has an associated barcode $\B^d(H)$, which we call the \textit{degree $d$ barcode of $H$}.  Theorem 6.2 of \cite{Uzi} asserts that any $I_{\alpha}$ for $(I_{\alpha}, m_{\alpha}) \in \B^d(H)$ will have a degree $d$ action as its left-hand endpoint and a degree $d+1$ action (or infinity) as its right-hand endpoint; we say that two actions of degrees $d$ and $d+1$ \textit{pair} with each other if they are endpoints of the same interval in $\B^d(H)$.  Combining Proposition 5.5, Theorem 6.2, and the beginning of the proof of Theorem 12.3 of \cite{Uzi} gives that every degree $d$ action $c$ of $H$ will appear as an endpoint of $\B^d(H) \cup \B^{d-1}(H)$ with multiplicity equal to the number of elements $[x(t), v] \in \widetilde{P}_d(H)$ such that $\A_H([x(t), v]) = c$.
\end{remark}

\begin{remark}
Since the finite-valued degree $d$ actions of $H$ comprise the left-hand endpoints of $\B^d(H)$, and since our interest in persistence modules and barcodes lies only in their application to this context of Hamiltonian Floer thoery, all barcodes $\B$ will be assumed from now on to have finite-valued left-hand endpoints.
\end{remark}

Given a barcode $\B = \{ (I_{\alpha}, m_{\alpha}) \}$, create a set of indexed intervals $\langle \B \rangle = \{ I_{\alpha}^{i_{\alpha}} \}_{1 \leq i_{\alpha} \leq m_{\alpha}}$ which treats an interval $I_{\alpha}$ with multiplicity $m_{\alpha}$ as $m_{\alpha}$ separate copies of $I_{\alpha}$.  For $\eps >0$ and a barcode $\B$, let $\langle \B \rangle_{\eps}$ denote the subset of $\langle \B \rangle$ consisting of all intervals of length less than or equal to $2\eps$.  A function $\mu$ from a subset of $\langle \B \rangle$ to a subset of $\langle \mathcal{C} \rangle$ is called an \textit{$\eps$-matching} between barcodes $\B$ and $\mathcal{C}$ if:

\begin{itemize}
\item $\langle \B \rangle \backslash \langle \B \rangle_{\eps}$ is contained in the domain of $\mu$.

\item $\langle \mathcal{C} \rangle \backslash \langle \mathcal{C} \rangle_{\eps}$ is contained in the image of $\mu$.

\item If $\mu ([a,b)) = [a',b')$, where $[a,b) \in \langle \B \rangle \backslash \langle \B \rangle_{\eps}$ or $[a',b') \in  \langle \mathcal{C} \rangle \backslash \langle \mathcal{C} \rangle_{\eps}$, then $|a - a'| < \eps$ and $b$ and $b'$ are either both infinity or both finite with $|b-b'| < \eps$.

\end{itemize}

Finally, the \textit{bottleneck distance} $d_b$ between barcodes $\B$ and $\mathcal{C}$ is defined as
\[
d_b(\B, \mathcal{C}) = \text{inf}\{\eps > 0 \, | \, \text{there exists an $\eps$-matching between $\B$ and $\mathcal{C}$} \}.
\]

In our context of Hamiltonian Floer theory, we have the following result, which is a much weaker version of Theorem 12.2 from \cite{Uzi}.

\begin{theorem}\label{BarCont}  Let $H_0$ and $H_1$ be two non-degenerate Hamiltonians on closed and symplectic $M$.  Then for any degree $d$,
\[
d_b(\B^d(H_0), \B^d(H_1)) \leq \smallint_0^1||H_0(t, \cdot) - H_1(t, \cdot)||_{L^{\infty}} dt.
\]
\end{theorem}

Though barcodes so far have only been defined for non-degenerate Hamiltonians, the above theorem may occasionally be applied to define barcodes for degenerate or even (as we will do later) merely continuous functions on $M$.

\end{subsection}

\begin{subsection}{Boundary depth}
The boundary depth of a Hamiltonian diffeomorphism is our main motivation for studying barcodes, so we pause very briefly to remind the reader of its definition, its relation to barcodes, and a few of its key properties.  See \cite{UsherFirst}, \cite{MikeSolo}, and \cite{Uzi} for more details.

Let $\phi$ be a Hamiltonian diffeomorphism generated by non-degenerate $H$.  After constructing its Floer chain complex $CF_*(H)$, we may define the quantities $\beta_d(\phi) \in \R$ as 
\[
\beta_d(\phi) = \underset{0 \neq x \in \partial(CF_d(H))}{\text{sup}} \text{inf}\{ \ell(y) - \ell(x) \, | \, \partial_H(y) = x \},
\]
which is independent of the choice of such an $H$.  The \textit{boundary depth} $\beta(\phi)$ of $\phi$ may be defined as
\[
\beta(\phi) = \underset{d \in \Z}{\text{sup}} \, \,  \beta_d(\phi)
\]
and is a finite quantity.  The relationship between $\beta(\phi)$ and the barcodes $\B^d(H)$ becomes clear when one deduces from Theorems 4.11 and 6.2 of \cite{Uzi} that $\beta_d(\phi)$ is simply the length of the longest finite-length bar in $\B^d(H)$.  The quantity $\beta(\phi)$ is of particular interest to us because it gives a lower bound on $\phi$'s Hofer norm; we refer the reader to \cite{EggBeaters} and \cite{MikeSolo} for previous instances in which the boundary depth is used to answer questions of Hofer's geometry.

Similar to our continuity result for barcodes, we have the following continuity result for boundary depth which will prove useful in a later argument.

\begin{theorem}\label{theorem-barely} (\cite{UsherFirst}, \cite{MikeSolo})  For a Hamiltonian $H$, set
\[
||H|| = \int_0^1 \left( \underset{M}{\text{max}}(H_t) - \underset{M}{\text{min}}(H_t) \right) \, dt.
\]
If $\phi, \psi \in Ham(M, \omega)$ are generated by non-degenerate Hamiltonians $H$ and $K$, respectively, then
\[
|\beta(\phi) - \beta(\psi)| \leq ||H - K||.
\]
\end{theorem}

It is this continuity result that allows us to define the boundary depth of a degenerate $\phi \in Ham(M, \omega)$ (namely, if $\phi$ is generated by $H$, choose a sequence of non-degenerate $H_k$'s which $C^0$-converges to $H$ and let $\beta(\phi)$ be the limit of the $\beta(\phi_{H_k})$).

\end{subsection}

\end{section}

\begin{section}{Radially symmetric Hamiltonians}\label{RSH}

Suppose that we may symplectically embed a ball $B(2\pi R) = \big\{ (\vec{x}, \vec{y}) \, | \, \sum_i (x_i^2 + y_i^2) \leq 2R \big\}$ of radius $\sqrt{2R}$ into $M$.  Let $f(r)$ be a smooth function on $[0,R]$ which has vanishing derivatives of all orders (except for possibly the $0$-th) at $r = R$ and which has $f'(0)$ not an integer multiple of $2 \pi$.  Letting $z = (\vec{x}, \vec{y})$ be coordinates on our symplectic ball $B(2\pi R)$ with symplectic form $\sum_i dx_i \wedge dy_i$, we may define a smooth function $F: M \rightarrow \R$ by

\begin{displaymath}
   F(p) = \left\{
     \begin{array}{lr}
       f\left( \frac{|z|^2}{2} \right), & p = z \in B(2\pi R) \\
         & \\
       f(R), & \text{otherwise} ,
     \end{array}
   \right.
\end{displaymath}

and an easy calculation shows that the flow $\phi_F^t$ corresponding to this function is given by

\begin{displaymath}
   \phi_F^t(p) = \left\{
     \begin{array}{lr}
        e^{\sqrt{-1}f' ( \frac{|z|^2}{2} ) t}z, & p = z \in B(2\pi R) \\
         & \\
       p, & \text{otherwise.}
     \end{array}
   \right.
\end{displaymath}

Assuming for now that there are only a finite number of $r_i$ for which $f'(r_i)$ is an integer multiple of $2 \pi$, the above formula for our flow tells us that we will have an $S^{2n-1}$'s worth of periodic orbits at every radius equal to $\sqrt{2r_i}$, and any orbit (with capping contained in $B(2\pi R)$) in such an $S^{2n-1}$ family will have action
\[
f(r_i) - f'(r_i)r_i.
\]

We see that all points outside of $B(2\pi R)$ are constant periodic orbits of action precisely $f(R)$.  We will also have a constant periodic orbit occurring at the center of our symplectic ball whose action (with trivial capping) will be precisely $f(0)$. Our condition on $f'(0)$ implies that this capped orbit will be non-degenerate and, assuming $2 \pi l < f'(0) < 2 \pi (l+1)$, will have Conley-Zehnder index $-2ln - n$ (we leave it to the reader to arrive at this formula by following the reasoning provided in \cite{Oancea}, while keeping in mind that they compute the negative version of our $\mu_{CZ}$).

As can be seen by the presence of non-isolated periodic points, our $F$ is degenerate, so we perturb it to a non-degenerate $\tilde{F}$ for which we may construct a barcode.  A very specific perturbation is chosen as follows.

We start with the $S^{2n-1}$ families of periodic orbits.  Along with our assumption that all points $r_i$ where $f'(r_i)$ is an integer multiple of $2 \pi$ are isolated, we further assume that $f''(r_i) \neq 0$ for each such $r_i$ so that we may perform the standard perturbation of $F$ around the $S^{2n-1}$ families of periodic orbits (see \cite{CFHW}, \cite{Oancea}, \cite{Sey}). In particular, define a perfect Morse function $h_{i}$ on the $S^{2n-1}_i$ corresponding to $r_i$, and smoothly extend it to a small tubular neighborhood in $B(2\pi R)$.  Calling these extended functions $h_{r_i}$, the time-dependent function $F + \delta \sum_i h_{r_i} \circ (\phi_F^t)^{-1}$ with $\delta$ small enough will have each $S^{2n-1}_i$ splitting into two periodic orbits $z_1, z_2$.  If $f'(r_i) = 2 \pi l$ and $[z_j, v_j]$ denotes these orbits with cappings contained in $B(2\pi R)$, their indices will be given by

\begin{displaymath}
   \mu_{cz}([z_1, v_1]) = \left\{
     \begin{array}{lr}
       -2ln+n, & f''(r_i)<0 \\
       -2ln+n-1, & f''(r_i)>0
     \end{array}
   \right.
\end{displaymath}

\begin{displaymath}
   \mu_{cz}([z_2, v_2]) = \left\{
     \begin{array}{lr}
       -2ln-n+1, & f''(r_i)<0 \\
       -2ln-n, & f''(r_i)>0
     \end{array}
   \right.
\end{displaymath}
Their actions will be approximately $f(r_i) - f'(r_i)r_i$, with the error term going to zero with $\delta$.

We now deal with the periodic orbits outside of $B(2\pi R)$.  Choose once and for all a Morse function $g: M \rightarrow [-1,0]$ that has a unique critical point (a maximum, where $g$ attains the value $0$) in $B(2\pi R)$, and choose a sufficiently small collar neighborhood $C$ of $\partial ( B(2\pi R) )$ so that $C$ contains no periodic points of $F$ which occur in $\text{int}(B(2\pi R))$; the existence of such a $C$ is guaranteed by our finiteness assumption on the number of $r_i$.  Then define $\tilde{g}$ to be equal to $g$ on $M \backslash (B(2\pi R) \cup C)$, $0$ on $B(2\pi R) \backslash C$, and to be smoothly extended to all of $M$ so that it has no critical points in $C$.  If $f$ is decreasing right before $R$, then the final step in our perturbation of $F$ will be to $ F + \delta \sum_i h_{r_i} \circ (\phi_F^t)^{-1} + \eps\tilde{g}$, with $\eps$ small enough so that the only periodic points in $M \backslash \text{int}(B(2\pi R))$ of this new function are critical points of $g$.  (If $f$ is increasing right before $R$, our final perturbation is instead to $F + \delta \sum_i h_{r_i} \circ (\phi_F^t)^{-1} - \eps\tilde{g}$.)  We refer to these periodic orbits as \textit{exterior orbits}.  By our choice of $g$, the exterior orbits (with trivial cappings) of $F + \delta \sum_i h_{r_i} \circ (\phi_F^t)^{-1} + \eps\tilde{g}$ will have Morse indices lying in  $[0, 2n-1]$ so that their Conley-Zehnder incices will lie in $[-n, n-1]$ by our convention (assuming again that $\eps$ is small enough).  Their actions will lie in $[f(R) -\eps, f(R))$.  (Such trivially capped orbits will have Conley-Zehnder indices in $[-n+1, n]$ and actions in $(f(R), f(R) + \eps]$ if our final perturbation was instead to $F + \delta \sum_i h_{r_i} \circ (\phi_F^t)^{-1} - \eps\tilde{g}$.)

We let $\tilde{F} = F + \delta \sum_i h_{r_i} \circ (\phi_F^t)^{-1} + \eps \tilde{g}$ denote this non-degenerate Hamiltonian.  To get all possible actions and indices of $\tilde{F}$'s capped periodic orbits, we must only consider the actions and indices already described and how they change under recappings.  Our monotonicity condition implies that any such change can only occur when $N \neq 0$, in which case increasing the index by $k 2N$ (with $k \in \Z$) via recapping will increase its action by $k \sigma(\lambda) \gamma$ (where $\sigma(\lambda)$ is as in Section \ref*{HFH}).  As noted earlier, these indexed actions give us all finite-valued endpoints of all bars in $\tilde{F}$'s barcode.

Moving our focus away from smooth functions, suppose that $f: [0, R] \rightarrow \R$ is piecewise linear.  We say that $f$ satisfies the \textit{slope condition} if all of its slopes are not integer multiples of $2 \pi$ and if the slope $s$ going into the line $r = R$ satisfies $|s| < 2 \pi$.  Assuming $f$ satisfies the slope condition, and letting $F: M \rightarrow \R$ be the $C^0$ function induced by $f$, our goal now is to show how we may associate to this non-differentiable $F$ a barcode in any degree.

We first describe a specific kind of perturbation of $F$, which we will refer to as \textit{standard} (or more commonly as a \textit{standard perturbation of $f$}); this perturbation is the same as that described in \cite{Sey}.  Pick small enough $\eps'$-neighborhoods around the $r$-values where $f$ is not differentiable so that no two neighborhoods intersect, and pick a smoothing $f_{\eps'}$ which has strictly monotonic first derivative on these $\eps'$ neighborhoods and which is equal to $f$ elsewhere.  (We choose our smoothing at $r = R$ so that our function $f_{\eps'}$ has vanishing derivatives of all orders, except possibly the $0$-th, at $r = R$.)  Where $F_{\eps'}$ is the function on $M$ induced by $f_{\eps'}$, choose $\eps$ and $\delta$ small enough to construct the non-degenerate, time-dependent perturbation $\tilde{F}_{\eps'}$ of $F_{\eps'}$ as described above:

\[
\tilde{F}_{\eps'} = F_{\eps'} + \delta \sum_i h_{r_i} \circ (\phi_F^t)^{-1} + \eps\tilde{g}.
\]

If $\eps'_k$ is a sequence converging to zero, we may choose similar smoothings $f_{\eps'_k}$ and appropriate sequences $\delta_k$, $\eps_k$ (both converging to zero) to create a sequence of standard perturbations $\tilde{F}_{\eps'_k}$ (abbreviated as $\tilde{F}_k$) of $F$ which $C^0$-converges to $F$ when $F$ is regarded as a function with domain $\R / \Z \times M$.  Our assumption that $|s|< 2 \pi$  ensures the existence of a collar neighborhood $C$ such that, for any $\eps'_k$ small enough, $C \cap \text{int}(B(2\pi R))$ contains no periodic orbits of $F_{\eps'_k}$.  We may therefore use the same function $\tilde{g}$ for every entry in our sequence $\tilde{F}_k$, a fact which will aid us momentarily.

Letting $d$ be any degree, each $\tilde{F}_k$ has the same number of actions in degree $d$ by the monotonic behavior of each $f_{\eps_k'}$'s derivative on the $\eps_k'$ neighborhoods.  Furthermore, the set of degree $d$ actions for $\tilde{F}_k$ forms a sequence converging to a specific set of real numbers.  (To see why these statements are true, let $\bar{r}$ be a point of non-differentiability for $f$ with $s_1$ and $s_2$ being the slopes of $f$ immediately before and after $\bar{r}$, and suppose $2 \pi l$ for some $l \in \Z$ is between $s_1$ and $s_2$.  By our choice of smoothings, every $f_{\eps_k'}$ has a unique $r$ value $r_{i,k}$ in $(\bar{r} - \eps'_k, \bar{r} + \eps'_k)$ for which $f_{\eps_k'}$'s derivative is $2 \pi l$.  Letting $[x, v]_k$ be the corresponding capped orbit of $\tilde{F}_k$ of lower (or higher) index $d$ (with capping contained in $B(2\pi R)$), we form the sequence of degree $d$ actions $\A([x,v]_k)$, which converges to $-2 \pi l \bar{r} + f(\bar{r})$.  Such convergence statements clearly apply to recappings of the $[x,v]_k$, as well as actions coming from the $y$-intercept and from exterior orbits, since we are using the same function $\tilde{g}$ for every entry $\tilde{F}_k$ in our limiting sequence.)  This fact is essential in proving the following:

\begin{claim}\label{degBarCode}  Abbreviate $\B^d(\tilde{F}_k)$ as $\B^d_k$.  The sequence $\B^d_k$ converges in the bottleneck distance to a unique barcode $\B^d$.
\end{claim}

\begin{proof}  

%We describe the construction of $\B^d$, after which it becomes clear that it is the limit of the $\B^d_k$.  Suppose that $a$ is a limit of degree $d$ actions, while $b$ is a limit of degree $d+1$ actions with $b \neq a$.  (In particular, we assume that $b$ is finite, though similar reasoning may be applied to the case that $b$ is infinite.)  Then for $\eps > 0$, define the integer $m^{\eps}_k(a,b)$ to be the number of bars $[a', b')$ in $\langle \B^d_k \rangle$ with $|a - a'| < \eps$ and $|b-b'|< \eps$.  When we choose $\eps$ so that $4 \eps$ is less than the minimal positive distance between all points in our limiting set of degree $d$ and $d+1$ actions, the sequence $m^{\eps}_k(a,b)$ will become constant for all $k$ big enough; denote this limit by $m(a,b)$.  Then define $\B^d$ to be the collection $\cup_{i,j} ( [a_i,b_j), m(a_i,b_j) )$ such that $m(a,b) \neq 0$. 

Let $A = \{ a_i \}_{i = 1}$ be the limiting set of degree $d$ actions, and let $B = \{ b_j \}_{j = 0}$ be the limiting set of degree $d+1$ actions unioned with $\{ \infty \}$ (set $b_0 = \infty$).  Choose $\eps$ so that $4\eps$ is less than the minimal positive distance between all elements of $A \cup (B \backslash \{ \infty \} )$.  For fixed elements $a_{i} \in A$ and $b_{j} \in B$, define the integer $m^{\eps}_k(a_i,b_j)$ to be the number of bars $[a', b')$ in $\langle \B^d_k \rangle$ with $|a_i - a'| < \eps$ and either $|b_j - b'|< \eps$ or $b' = \infty$ in the case that $j = 0$.

%We first claim that when we choose $\eps$ so that $4 \eps$ is less than the minimal positive distance between all elements of $\{ a_i \}_{i=1}^{N_d} \cup \{ b_j \}_{j=1}^{N_{d+1}}$, the sequence $m^{\eps}_k(a_i,b_j)$ will become constant for all $k$ big enough.  Indeed, our sequence of functions $\tilde{F}_k$ is Cauchy with respect to the $C^0$ norm, implying the existence of $\eps/2$-matchings $\mu_{\eps/2}^{k_1, k_2}$ between $\B^d_{k_1}$ and $\B^d_{k_2}$ for all $k_1, k_2$ big enough.  Meanwhile, 

For all $k$ big enough, every finite-valued endpoint of $\B^d_k$ is contained within the union of intervals $\cup_{i, j \neq 0} \{ (a_i - \eps, a_i + \eps), (b_j - \eps, b_j + \eps)\}$, and since our sequence of functions $\tilde{F}_k$ is Cauchy with respect to the $C^0$ norm, we may assert the existence of $\eps$-matchings $\mu_{\eps}^{k_1, k_2}$ between $\B^d_{k_1}$ and $\B^d_{k_2}$ for all $k_1, k_2$ big enough.  Moreover, for such $k_1$ and $k_2$, the bars in $\langle \B^d_{k_1} \rangle$, $\langle \B^d_{k_2} \rangle$ which define $m^{\eps}_{k_1}(a_i, b_j)$, $m^{\eps}_{k_2}(a_i, b_j)$, are of length at least $2 \eps$ when $a_i \neq b_j$; these bars are therefore in the domains and ranges of our $\mu_{\eps}^{k_1, k_2}$.  From this, we deduce that the sequence $m^{\eps}_k(a_i, b_j)$ is eventually constant and so converges to some integer $m^{\eps}(a_i, b_j)$ when $a_i \neq b_j$.  Define $\B^d$ to be the collection
\[
\left \{ ( [a_i, b_j), m^{\eps}(a_i,b_j) ) \, \, \middle| \, \, \, a_i \in A, \, b_j \in B, \, a_i \neq b_j, \, m^{\eps}(a_i, b_j) \neq 0 \right \}.
\]

From here, it is easy to conclude that $\B^d$ is in fact the limit of the $\B^d_k$.  Indeed, let $\eps' >0$ be less than $\eps$.  Then for any $k$ large enough, there is clearly an injection $\mu_{\eps'}^k:  \langle \B^d \rangle \rightarrow \langle \B^d_k \rangle$  so that

\begin{itemize}
\item  $\mu_{\eps'}^k$ satisfies the third condition of being an $\eps'$-matching (see Section \ref*{BC}).

\item  any bar not in the range of $\mu_{\eps'}^k$ has endpoints contained in an interval of the form $(c - \eps', c + \eps')$, where $c \in A \cup (B \backslash {\infty})$.

\end{itemize}
(The second condition holds since $\eps' < \eps$.)  In particular, $\mu_{\eps'}^k$ is an $\eps'$-matching.

\end{proof}

Letting $H_k$ be any other sequence of non-degenerate Hamiltonians which $C^0$-converge to $F$ gives another sequence of barcodes $\B^d(H_k)$ which must also necessarily converge to $B^d$; assuming otherwise, we could fix $k'$ big enough and compare $\B^d(H_{k'})$ with $\B^d_{k'}$ from the proof of Claim \ref{degBarCode} to arrive at a contradiction of the continuity of barcodes.  Our function $F$ may therefore be attributed a well-defined barcode in any degree $d$, though we abuse notation and refer to it as the degree $d$ barcode of $f$, or $\B^d(f)$.  It is clear from our construction of $\B^d(f)$ and Theorem \ref{BarCont} that for two piecewise linear functions $f_1$ and $f_2$ satisfying our slope condition, we have
\[
d_b(\B^d(f_1), \B^d(f_2)) \leq ||f_1 - f_2 ||_{L^{\infty}}.
\]

%\begin{remark}\label{FINE}
%It is clear that a similar and simpler process may be used to construct well-defined barcodes for any radially symmetric Hamiltonian supported in $B(2\pi R)$, and that the above inequality is also satisfied when $f_1$ is piecewise linear and $f_2$ is smooth (or both are smooth).
%\end{remark}

We pause to define some terms.  In the following definitions, $f$ refers to a piecewise linear function satisfying our slope condition, $\{ r_i \}_{i \geq 0}$ are the $r$-values of $f$'s points of non-differentiability in decreasing order with $r_0 = R$, and $\{ m_i \}_{i \geq 1}$ are the slopes of $f$ as we move from right to left (so $m_{i+1}$ and $m_i$ are the slopes on the left and right, respectively, of the point $(r_i, f(r_i))$).

\begin{dfn}  A number $c \in \R$ is a \textit{degree $d$ action of $f$} if it is the limit of a sequence of degree $d$ actions arising from a sequence of standard perturbations of $f$.
\end{dfn}

\begin{dfn}  The \textit{degree $d$ action spectrum of $f$ with multiplicity}, denoted by $\textit{Spec}^d_m(f)$, is the collection of all degree $d$ actions of $f$ considered with multiplicity.  Similarly, the \textit{action spectrum of $f$ with multiplicity} $\textit{Spec}_m(f)$ refers to the union over all degrees $d$ of the $\textit{Spec}_m^d(f)$.
\end{dfn}

In light of Section \ref*{HFH}, it is clear that right-hand endpoints of $\B^d(f)$ are either infinity or elements of $\textit{Spec}_m^{d+1}(f)$, while left-hand endpoints are elements of $\textit{Spec}_m^d(f)$.

%We call $c$ a \textit{degree $d$ action of $f$} if $c$ is the limit of a sequence of degree $d$ actions arising from a sequence of standard perturbations of $f$.  The collection of all degree $d$ actions, considered with multiplicity, of $f$ will be referred to as the \textit{degree $d$ action spectrum of $f$ with multiplicity} and denoted by $\textit{Spec}^d_m(f)$, while the \textit{action spectrum of $f$ with multiplicity} $\textit{Spec}_m(f)$ refers to the union over all degrees $d$ of the $\textit{Spec}^d_m(f)$.

\begin{dfn}
If $m_{i+1} < m_i$ (resp. $m_i < m_{i+1}$), then we call $(r_i, f(r_i))$ a \textit{concave up} (resp. \textit{down}) \textit{kink} of $f$.
\end{dfn}

By the comments immediately preceding Claim \ref{degBarCode}, $\textit{Spec}^d_m(f)$ and $\textit{Spec}_m(f)$ are well-defined, and we enumerate the elements of $\textit{Spec}_m(f)$ with their degrees below.

\begin{enumerate}
\item[(1)]  If $(r_i, f(r_i))$ is a concave up kink of $f$ with $m_{i+1} < 2 \pi l < m_i$ for some $l \in \Z$, then $-2 \pi l r_i + f(r_i)$ will be a degree $-2ln +n -1$ and a degree $-2ln - n$ action of $f$.  Furthermore, for any integer $k$, $-2 \pi l r_i + f(r_i) + k \sigma(\lambda) \gamma$ will be a degree $-2ln + n - 1 + k2N$ and a degree $-2ln -n + k2N$ action of $f$ if $N \neq 0$.

\item[(2)]  If $(r_i, f(r_i))$ is a concave down kink of $f$ with $m_{i+1} > 2 \pi l > m_i$ for some $l \in \Z$, then $-2 \pi l r_i + f(r_i)$ will be a degree $-2ln +n$ and a degree $-2ln - n + 1$ action of $f$.  Furthermore, for any integer $k$, $-2 \pi l r_i + f(r_i) + k \sigma(\lambda) \gamma$ will be a degree $-2ln + n + k2N$ and a degree $-2ln -n + 1 + k2N$ action of $f$ if $N \neq 0$.

\item[(3)]  If the slope $s$ of the line coming out of the $y$-axis satisfies $2 \pi l < s < 2 \pi (l+1)$, then $f(0)$ will be a degree $-2ln - n$ action of $f$; as before, for any integer $k$, $f(0) + k \sigma(\lambda) \gamma$ will be a degree $-2ln - n + k2N$ action of $f$ if $N \neq 0$.

\item[(4)]  If $g$ has a critical point of Morse index $j$ outside of $B(2\pi R)$, then $f(R)$ will be a degree $j-n$ action of $f$; as before, for any integer $k$, $f(R) + k\sigma(\lambda)\gamma$ will be a degree $j - n + k2N$ action of $f$ if $N \neq 0$.

\end{enumerate}

Note that a sequence of standard perturbations of $f$ might have some sequence of bars whose lengths go to zero as the sequence progresses, so there is no guarantee that any single action from the above enumeration has to appear in any $\B^d(f)$.  However, it should be clear from our construction of $\B^d(f)$ that if any degree $d$ action from the above enumeration has multiplicity one in $\textit{Spec}_m(f)$, then it must appear in either $\B^d(f)$ or $\B^{d-1}(f)$.

\begin{dfn}  A \textit{kink action} of $f$ is an action coming from either (1) or (2) in the above enumeration, while an \textit{exterior action} of $f$ is one coming from (4).
\end{dfn}

Our final piece of terminology is only to be applied in the case that $N \neq 0$, i.e. that $M$ is monotone but not symplectically aspherical.  Where $f$ satisfies our slope condition with  $\{ r_i \}_{i \geq 0}$ and $\{ m_i \}_{i \geq 1}$ as before, let $S^i$ be the collection of integers $l$ with $2 \pi l$ between $m_i$ and $m_{i+1}$.  

\begin{dfn}  If $N$ and $\gamma$ are both non-zero, we say that $f$ has \textit{distinct kink actions} if

\begin{enumerate}

\item[(1a)]  for any two triples $(r_i, l, k)$ and $(r_{i'}, l', k')$, with $l \in S^{i}$, $l' \in S^{i'}$, and $k, k' \in \Z$, we have the equalities
\[
r_{i} = r_{i'}, \, \, l = l', \, \, k = k'
\]
holding whenever
\[
-2 \pi l r_{i} + f(r_{i}) + k \sigma(\lambda) \gamma = -2 \pi l' r_{i'} + f(r_{i'}) + k' \sigma(\lambda) \gamma \, ;
\]

\item[(1b)] for any triple $(r_i, l, k)$ with $l \in S^i$ and $k \in \Z$, $-2 \pi l r_{i} + f(r_{i}) + k \sigma(\lambda) \gamma$ does not equal $f(0) + k' \sigma(\lambda) \gamma$ or $f(R) + k' \sigma(\lambda) \gamma$ for any integer $k'$.

\end{enumerate}

\vspace{0.5cm}

In the case that $N \neq 0$ and $\gamma = 0$, we say that $f$ has distinct kink actions if 

\begin{enumerate}

\item[(2a)]  for any two pairs $(r_{i}, l)$ and $(r_{i'}, l')$ with $l \in S^{i}$, $l' \in S^{i'}$, we have the equalities
\[
r_{i} = r_{i'}, \, \, l = l'
\]
holding whenever
\[
-2 \pi l r_{i} + f(r_{i})= -2 \pi l' r_{i'} + f(r_{i'});
\]

\item[(2b)]  for any pair $(r_i, l)$ with $l \in S^i$, $-2 \pi l r_{i} + f(r_{i})$ does not equal $f(0)$ or $f(R)$.

\end{enumerate}
\end{dfn}
\vspace{0.5cm}

Conditions (1b) and (2b) ensure that no kink action equals any exterior action or any action coming from the $y$-axis.

With our terminology established, we may conclude this section with a few key lemmas and theorems concerning barcodes of piecewise linear functions.

\begin{lemma}\label{lemma-very-useful}  Let $\eps > 0$ be given, and let $f_1$ and $f_2$ be two piecewise linear functions satisfying our slope condition and the following:

\begin{itemize}

\item $||f_1 - f_2||_{L^{\infty}} < \eps.$

\item the minimal distance between finite $a$ and any action of $f_1$ or $f_2$ outside of $I_{\eps}(a) := (a - \eps, a + \eps)$ is at least $3\eps$.

\end{itemize}

Then for a fixed degree $d$, the number of degree $d$ actions in $I_{\eps}(a)$ which pair with degree $d+1$ actions outside of $I_{\eps}(a)$ is the same for $f_1$ and $f_2$; this conclusion with $d+1$ replaced by $d-1$ also holds.

\end{lemma}

\begin{proof}  The proof of either implication is the same, so we restrict our attention to the first.  By the assumption that $||f_1 - f_2||_{L^{\infty}} < \eps$, we know that an $\eps$-matching $\mu_{\eps}$ exists between $\B^d(f_1)$ and $\B^d(f_2)$.  Note that any pairing between a degree $d$ action in $I_{\eps}(a)$ with a degree $d+1$ action outside of $I_{\eps}(a)$ gives rise to a bar of length at least $2 \eps$ and so is in the domain (or range) of $\mu_{\eps}$.  Moreover, our second condition implies that $\mu_{\eps}$ must match such a bar to a bar whose degree $d$ (resp. $d+1$) endpoint also lies inside (resp. outside) of $I_{\eps}(a)$.  Hence, $\mu_{\eps}$ gives a bijection between the set of intervals of the form $[c^d, c^{d+1})$, with $c^d \in I_{\eps}(a)$ and $c^{d+1} \notin I_{\eps}(a)$, for $\B^d(f_1)$ and the set of such intervals for $\B^d(f_2)$.

\end{proof}

Lemma \ref{lemma-very-useful} is helpful in proving the following theorem, which is key to proving Theorem \ref{main-theorem}.  Before proving Theorem \ref{cd-n+1-solo} in full generality, however, we prove it in the case of $f$ having distinct kink actions; Theorem \ref{cd-n+1-solo} applied to this case is expressed as Lemma \ref{lemma-cd-n+1}.

\begin{theorem}\label{cd-n+1-solo}
Let $f$ be any piecewise linear function satisfying our slope condition, and let $c^{n+1}$ be an action which in degree $n+1$ only comes from concave down kinks of $f$.  Then $c^{n+1}$ does not enter into $\B^{n+1}(f)$, and if no degree $n$ action equals $c^{n+1}$, then $c^{n+1}$ must appear in $\B^n(f)$.
\end{theorem}

\begin{lemma}\label{lemma-cd-n+1}  Let $f$ be a piecewise linear graph satisfying our slope condition and having distinct kink actions.  Let $c^{n+1}$ denote a degree $n+1$ action coming from a concave down kink in $f$'s graph.  Then $c^{n+1}$ must appear in $\B^n(f)$.

\end{lemma}

\begin{proof}[ Proof of Lemma \ref{lemma-cd-n+1}]
We restrict our attention to the case of $N \neq 0$, since nearly identical (and even simpler) reasoning applies to the case of $N = 0$.  We also must separate our proof into the cases that $\lambda \neq 0$ (so $\gamma \neq 0$) and $\lambda = 0$ (so $\gamma = 0$ but $M$ is not symplectically aspherical).

\vspace{0.5cm}

\textbf{The case that $\lambda \neq 0$}.

Let $f$ be such a function with $\{ r_i \}_{i \geq 0}$ and $\{ m_i \}_{i \geq 1}$ as previously defined.  Our goal is to choose an appropriate homotopy ending in $f$ for which it will be easy to keep track of the corresponding continuum of barcodes.  Our homotopy of choice is performed by connecting the zero function to $f$ through the series of intermediate functions $g_i$ defined by

\begin{displaymath}
   g_i(r) = \left\{
     \begin{array}{lr}
       f(r), & r \geq r_i\\
       f(r_i) + m_i (r - r_i), & 0 \leq r \leq r_i  
     \end{array}
   \right.
\end{displaymath}
for $i \geq 1$.  We connect these intermediate functions via straight-line homotopies
\[
h_i(t, r) = tg_i + (1-t)g_{i-1},
\]
where we take $g_0$ to be the zero function, and we call the concatenation of these homotopies $h_t$.  Geometrically, this homotopy is taking the graph of the zero function and folding it along the kinks of $f$'s graph from the outside in until $f$'s graph is created (see Figures \ref*{first} - \ref*{fourth}).

\begin{subfigures}
\begin{figure}

\begin{tikzpicture}[scale=.4]\label{bending-one}
		\draw[<->] (-16,0) -- (16,0) node[right] {$r$};
		\draw[dashed] (-4,0.25) -- (-1,4);
		\draw[dashed] (-1, 4) -- (0,0) -- (1,4);
		\draw[dashed] (-10,-0.25) -- (-8.5,-3) -- (-8,0) -- (-7.5,-3) -- (-6,-0.25) -- (-4,0.25);
		\draw[dashed] (-16, 0) -- (-13, 0.25) -- (-12.25, 3.5) -- (-12, 0)-- (-11.75, 3.5)-- (-11, 0.25)-- (-10, -0.25);
		\draw[dashed] (1,4) -- (4,0.25) -- (4, 0.25) -- (6, 0) -- (10, 2) -- (11,0) -- (12, 2) -- (16, 0.25);
		\draw[very thick] (-12/1.75 +12, 5) -- (16, 0.25);
		\draw[dotted] (-16, 5) -- (16, 5);
        \draw[dotted] (-16, -4) -- (16, -4);
\end{tikzpicture}

\caption{One of our functions $g_i$, with the graph of $f$ represented by the dashed lines.}

\label{first}

\vspace{1cm}

\begin{tikzpicture}[scale=.4]\label{bending-two}
		\draw[<->] (-16,0) -- (16,0) node[right] {$r$};
		\draw[dashed] (-4,0.25) -- (-1,4);
		\draw[dashed] (-1, 4) -- (0,0) -- (1,4);
		\draw[dashed] (-10,-0.25) -- (-8.5,-3) -- (-8,0) -- (-7.5,-3) -- (-6,-0.25) -- (-4,0.25);
		\draw[dashed] (-16, 0) -- (-13, 0.25) -- (-12.25, 3.5) -- (-12, 0)-- (-11.75, 3.5)-- (-11, 0.25)-- (-10, -0.25);
		\draw[dashed] (1,4) -- (4,0.25) -- (4, 0.25) -- (6, 0) -- (10, 2) -- (11,0) -- (12, 2) -- (16, 0.25);
		\draw[very thick] (12, 2) -- (16, 0.25);
        \draw[very thick] (9, -4) -- (12, 2);
        \draw[gray] (-16, 5) -- (12,2);
	   	\draw[gray] (-10, 5) -- (12,2);
	   	\draw[gray] (-6, 5) -- (12,2);
        \draw[gray] (-4, 5) -- (12,2);
	   	\draw[gray] (-2, 5) -- (12,2);
	   	\draw[gray] (0, 5) -- (12,2);
	   	\draw[gray] (1, 5) -- (12,2);
		\draw[gray] (2, 5) -- (12,2);
		\draw[gray] (3,5) -- (12,2);
		\draw[gray] (4,5) -- (12, 2);
		\draw[gray] (-16,4) -- (12,2);
		\draw[gray] (-16,3) -- (12,2);
		\draw[gray] (-16,2) -- (12,2);
		\draw[gray] (-16,1) -- (12,2);
		\draw[gray] (-16,0) -- (12,2);
		\draw[gray] (-16, -1) -- (12, 2);
        \draw[gray] (-16, -2) -- (12, 2);
        \draw[gray] (-16, -3) -- (12,2);
        \draw[gray] (-16, -4) -- (12, 2);
        \draw[gray] (-12, -4) -- (12, 2);
        \draw[gray] (-8, -4) -- (12, 2);
        \draw[gray] (-5, -4) -- (12, 2);
        \draw[gray] (-3, -4) -- (12, 2);
        \draw[gray] (-1, -4) -- (12, 2);
        \draw[gray] (0, -4) -- (12, 2);
        \draw[gray] (1, -4) -- (12, 2);
        \draw[gray] (2, -4) -- (12, 2);
        \draw[gray] (3, -4) -- (12, 2);
        \draw[gray] (4, -4) -- (12, 2);
        \draw[gray] (5, -4) -- (12, 2);
        \draw[gray] (6, -4) -- (12, 2);
        \draw[gray] (7, -4) -- (12, 2);
        \draw[gray] (8, -4) -- (12, 2);
        \draw[gray] (9, -4) -- (12, 2);
        \draw[dotted] (-16, 5) -- (16, 5);
        \draw[dotted] (-16, -4) -- (16, -4);
        \draw[gray] (-12/1.75 +12, 5) -- (16, 0.25);
\end{tikzpicture}

\caption{The function $g_i$ bending down at $(r_{i+1}, f(r_{i+1}))$ to make $g_{i+1}$.}

\label{second}

\vspace{1cm}

\begin{tikzpicture}[scale=.4]\label{bending-three}
		\draw[<->] (-16,0) -- (16,0) node[right] {$r$};
		\draw[dashed] (-4,0.25) -- (-1,4);
		\draw[dashed] (-1, 4) -- (0,0) -- (1,4);
		\draw[dashed] (-10,-0.25) -- (-8.5,-3) -- (-8,0) -- (-7.5,-3) -- (-6,-0.25) -- (-4,0.25);
		\draw[dashed] (-16, 0) -- (-13, 0.25) -- (-12.25, 3.5) -- (-12, 0)-- (-11.75, 3.5)-- (-11, 0.25)-- (-10, -0.25);
		\draw[dashed] (1,4) -- (4,0.25) -- (4, 0.25) -- (6, 0) -- (10, 2) -- (11,0) -- (12, 2) -- (16, 0.25);
		\draw[very thick] (12, 2) -- (16, 0.25);
        \draw[very thick] (9, -4) -- (12, 2);
        \draw[dotted] (-16, 5) -- (16, 5);
        \draw[dotted] (-16, -4) -- (16, -4);
\end{tikzpicture}

\caption{The function $g_{i+1}$.}

\label{third}

\vspace{1cm}

\begin{tikzpicture}[scale=.4]\label{bending-four}
		\draw[<->] (-16,0) -- (16,0) node[right] {$r$};
		\draw[dashed] (-4,0.25) -- (-1,4);
		\draw[dashed] (-1, 4) -- (0,0) -- (1,4);
		\draw[dashed] (-10,-0.25) -- (-8.5,-3) -- (-8,0) -- (-7.5,-3) -- (-6,-0.25) -- (-4,0.25);
		\draw[dashed] (-16, 0) -- (-13, 0.25) -- (-12.25, 3.5) -- (-12, 0)-- (-11.75, 3.5)-- (-11, 0.25)-- (-10, -0.25);
		\draw[dashed] (1,4) -- (4,0.25) -- (4, 0.25) -- (6, 0) -- (10, 2) -- (11,0) -- (12, 2) -- (16, 0.25);
		\draw[very thick] (12, 2) -- (16, 0.25);
        \draw[very thick] (11, 0) -- (12, 2);
        \draw[gray] (11, 0) -- (7, -4);
        \draw[gray] (11, 0) -- (5, -4);
        \draw[gray] (11, 0) -- (3, -4);
        \draw[gray] (11, 0) -- (1, -4);
        \draw[gray] (11, 0) -- (-1, -4);
        \draw[gray] (11, 0) -- (-3, -4);
        \draw[gray] (11, 0) -- (-6, -4);
        \draw[gray] (11, 0) -- (-9, -4);
        \draw[gray] (11, 0) -- (-12, -4);
        \draw[gray] (11, 0) -- (-16, -4);
        \draw[gray] (-16, -3) -- (11, 0);
        \draw[gray] (-16, -2) -- (11, 0);
        \draw[gray] (-16, -1) -- (11, 0);
        \draw[gray] (-16, 1) -- (11, 0);
        \draw[gray] (-16, 2) -- (11, 0);
        \draw[gray] (-16, 3) -- (11, 0);
        \draw[gray] (-16, 4) -- (11, 0);
        \draw[gray] (-16, 5) -- (11, 0);
	   	\draw[gray] (-10, 5) -- (11,0);
	   	\draw[gray] (-6, 5) -- (11,0);
        \draw[gray] (-4, 5) -- (11,0);
	   	\draw[gray] (-2, 5) -- (11,0);
	   	\draw[gray] (0, 5) -- (11,0);
	   	\draw[gray] (1, 5) -- (11,0);
		\draw[gray] (2, 5) -- (11,0);
		\draw[gray] (3,5) -- (11,0);
		\draw[gray] (4,5) -- (11, 0);
		\draw[gray] (5, 5) -- (11,0);
	   	\draw[gray] (6, 5) -- (11,0);
		\draw[gray] (7, 5) -- (11,0);
		\draw[gray] (8,5) -- (11,0);
		\draw[gray] (9, -4) -- (12, 2);
        \draw[very thick] (11-5/2, 5) -- (11, 0);
        \draw[dotted] (-16, 5) -- (16, 5);
        \draw[dotted] (-16, -4) -- (16, -4);
\end{tikzpicture}

\caption{The function $g_{i+1}$ bending up at $(r_{i+2}, f(r_{i+2}))$ to make $g_{i+2}$.}

\label{fourth}

\end{figure}
\end{subfigures}

A few comments about the homotopy $h_t$ are in order.  Note that for all but finitely many values of time $T_0 = \{ t_{\alpha} \}$, each function $h_t$ satisfies our slope condition; the times it does not correspond to when the slope out of the $y$-axis is a multiple of $2 \pi$.  Hence, for any $t$ in an interval of the form $(t_{\alpha}, t_{\alpha +1})$, the function $h_t$ has a well-defined barcode.  Next, let $(r^{\alpha}, f(r^{\alpha}))$ be the point of non-differentiability for $f$ at which $h_t$ is bending for $t \in (t_{\alpha}, t_{\alpha}+1)$; then for all times in this interval, we see that the slope on the right of $(r^{\alpha}, f(r^{\alpha}))$ stays constant while the slope $s(t)$ (the slope of the line coming out of the $y$-axis at time $t$) on its left is between $2\pi l$ and $2\pi (l+1)$ for some integer $l$.  This, in conjunction with the point $(r^{\alpha}, f(r^{\alpha}))$ at which this kink occurs being stationary, implies that the set of all actions coming from this kink is the same for all such $h_t$, and the same is clearly true for all such actions coming from kinks $(r_i, f(r_i))$ with $r_i \geq r^{\alpha}$.  From this we can conclude that any change in $\textit{Spec}_m(h_t)$ with $t$ lying in $(t_{\alpha}, t_{\alpha+1})$ can only come from recappings of the $y$-intercept.  The degree of any such action does not change with time since $s(t)$ does not cross a multiple of $2 \pi$.  So for a fixed degree $d$, we may further conclude that $\# |\textit{Spec}^d_m(h_t)|$ stays the same as $t$ varies between $(t_{\alpha}, t_{\alpha + 1})$, and moreover, that the actions of $h_t$ may be parametrized as functions of time with domain $(t_{\alpha}, t_{\alpha + 1})$. Finally, our formulae for the possible degrees of actions coming from the $y$-intercept tell us that they are all of the same parity as $n$, so whenever $d$ has parity differing from $n$, $\textit{Spec}^{d}_m(h_t)$ is the same for all $t$ in $(t_{\alpha}, t_{\alpha + 1})$, i.e. these actions are constant as functions of time.

Now examine what happens at a time $t_{\alpha} \in T_0$.  For $t \in (t_{\alpha - 1}, t_{\alpha})$, we can parametrize the action (with trivial capping) coming from the $y$-intercept of $h_t$ as $h_t(0)$, while recappings of this action will be of the form $h_t(0) + k \sigma(\lambda) \gamma$ with $k \in \Z$.  This parametrization will also hold for times in $(t_{\alpha}, t_{\alpha + 1})$, though the degrees of these actions may differ.

Suppose for now that the function $s(t)$ is increasing, implying that $h_t(0)$ is decreasing and $(r^{\alpha}, f(r^{\alpha}))$ is a concave down kink for $f$.  If $s(t_{\alpha}) = 2 \pi l$ for $l \in \Z$, then $s(t)$ lies between $2 \pi (l -1)$ and $2 \pi l$ for $t \in (t_{\alpha - 1}, t_{\alpha})$, so our enumeration of actions and their degrees tells us that $h_t(0) + k \sigma(\lambda) \gamma$ has index $-2(l-1)n - n + k2N = -2ln+n + k2N$ and limits on $-2 \pi l r^{\alpha} + f(r^{\alpha}) + k \sigma (\lambda) \gamma$ as $t$ goes to $t_{\alpha}$.  Examining $h_t$ for times $t \in (t_{\alpha}, t_{\alpha + 1})$, we note that our kink at $(r^{\alpha}, f(r^{\alpha}))$ has an extra multiple of $2 \pi$ lying between the slopes on its left ($s(t)$) and right, so we have infinitely many new pairs of actions $\{c_{1,k}$, $c_{2,k} \}_{k \in \Z}$ with 

\begin{align*}
c_{1,k} = & -2 \pi l r^{\alpha} + f(r^{\alpha}) + k \sigma(\lambda) \gamma = h_{t_{\alpha}}(0) + k \sigma(\lambda)\gamma, \, \, \, \, \text{of degree} -2ln +n + k2N\\
c_{2, k} = & -2 \pi l r^{\alpha} + f(r^{\alpha}) + k \sigma (\lambda) \gamma = h_{t_{\alpha}}(0) + k \sigma(\lambda)\gamma, \, \, \, \, \text{of degree} -2ln-n+1 + k2N
\end{align*}
coming from the set of kinks in our graph.  In particular, note that $h_t(0) + k \sigma (\lambda) \gamma$ for $t \in (t_{\alpha - 1}, t_{\alpha})$ has index and limiting action equal to the index and action of $c_{1, k}$, while for times $t \in (t_{\alpha}, t_{\alpha + 1})$ it has index $-2ln - n + k2N$ and action limiting on  $-2 \pi l r^{\alpha} + f(r^{\alpha}) + k \sigma(\lambda) \gamma$ as $t$ decreases to $t_{\alpha}$.

Finally, note that for any $t \notin T_0$, the kink actions of $h_t$ are a subset of the kink actions of $f$.

With these observations about $h_t$ out of the way, we continue with our proof.  Set $T = [0,1] \backslash T_0$, and let $4\eps > 0$ be the smaller of the minimal positive distance between all of $f$'s kink actions and $\gamma$.  By our analysis of our homotopy, we know that the minimal distance between $h_t$'s kink actions will be greater than $4\eps$ for all $t \in T$ (where we consider the minimum of the empty set to be infinity, in this case).

Now let $t_0$ be the time of $c^{n+1}$'s inception in $\textit{Spec}_m^{n+1}(f)$.  Since $t_0 \neq 0$, we may find small enough intervals of time $(t_{-1}, t_0), (t_0, t_1) \subset T$ so that for any $t_- \in (t_{-1}, t_0)$ and $t_+ \in (t_0, t_1)$ we have $||h_{t_-} - h_{t_+}||_{L^{\infty}} < \eps$.  This in particular implies that $|h_{t_+}(0) - h_{t_-}(0)| < \eps$ for any pair of times $t_-$ and $t_+$.

So let $t_-$ and $t_+$ be any two such times.  With $t_0$ being the time of $c^{n+1}$'s inception, we must have $t_0 \in T_0$, so we can write $s(t_0) = 2 \pi l$.  The function $s(t)$ must be increasing on the interval $(t_{-1}, t_1)$ because $c^{n+1}$ comes from a concave down kink.  Hence, $s(t_+) > s(t_0)$, so the possible actions coming from the $y$-intercept at time $t_+$ will be of degree
\[
-2ln - n + k 2N
\]
and have the form
\[
h_{t_+}(0) + k \sigma(\lambda) \gamma.
\]

Moreover, since $t_0$ is the time of $c^{n+1}$'s inception and $s(t_0) = 2 \pi l$, we must have a solution $k$ to the equations
\[
-2ln -n + 1 + k 2N = n+1,
\]
and
\[
h_{t_0}(0) + k \sigma (\lambda) \gamma = c^{n+1}.
\]

This same value of $k$ will give us an action $h_{t_+}(0) + k \sigma(\lambda) \gamma$ of degree $n$ coming from the $y$-intercept.  By our choice of $\eps$, this degree $n$ action does not equal any other actions from $h_{t_+}$ and therefore exists in either $\B^{n-1}(h_{t_+})$ or $\B^n(h_{t_+})$.  The number of degree $n$ actions in $I_{\eps}(c^{n+1})$ at time $t_-$ is zero (see Lemma \ref{lemma-very-useful} to recall what $I_{\eps}(c^{n+1})$ denotes), so we may apply Lemma \ref{lemma-very-useful} to say that $h_{t_+}(0) + k \sigma (\lambda) \gamma$ must pair with an action in $I_{\eps}(c^{n+1})$ at time $t_+$.  The only other actions in $I_{\eps}(c^{n+1})$ at time $t_+$ are two actions equal to
\[
c^{n+1} = h_{t_0}(0) + k \sigma (\lambda) \gamma
\]
with one of degree $n+1$ and the other of degree $3n$.  So this degree $n$ action pairs with our degree $n+1$ action, implying that $c^{n+1}$ is an endpoint in $\B^n(h_{t_+})$ and therefore that Lemma \ref{lemma-cd-n+1} holds for any $h_t$ with $t \in (t_0, t_1)$.  See Figure \ref*{action_inception} for a depiction of this evolution of $\B^n(h_t)$, where in the picture for $\B^n(h_{t_+})$ appearing on the right, the red endpoint represents the degree $n$ action $h_{t_+}(0) + k \sigma (\lambda) \gamma$, while the blue (resp. lightly shaded) endpoint represents the degree $n+1$ (resp. $3n$) action $c^{n+1}$.  In the picture for $\B^n(h_{t_-})$, the lightly shaded endpoint represents the degree $3n$ action $h_{t_-}(0) + k \sigma(\lambda) \gamma.$

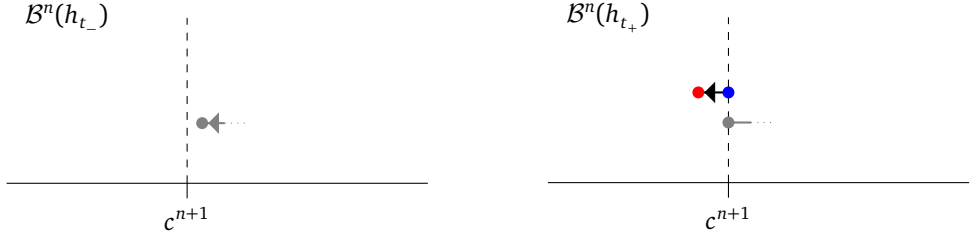
\begin{figure}[ht]
\vspace{1.5cm}
\hspace{2.5cm}
\begin{tikzpicture}[scale=.4]
		\node[] at (-8, 5.5) {$\B^n(h_{t_-})$};
		\draw[-] (-10, 0) -- (4, 0);
		\VertDash{-4}{$c^{n+1}$};
		\ParBarsLight{-3.5}{2}{black}{+};
		\draw[-triangle 90, black!50] (-3.5 + 0.45, 2) -- (-3.5 + 0.45 - 0.25, 2);
\end{tikzpicture}
\hspace{3cm}
\begin{tikzpicture}[scale=.4]
		\node[] at (-8, 5.5) {$\B^n(h_{t_+})$};
		\VertDash{-4}{$c^{n+1}$};
		\ParBarsLight{-4}{2}{black}{+};
		\GrowBarsLeft{-4}{3}{1}{blue}{red};
		\draw[-] (-10, 0) -- (4, 0);
\end{tikzpicture}

\vspace{0.5cm}

\caption{The evolution of $\B^n(h_t)$ with respect to time.  The lower bar in each picture does not appear in $\B^n(h_t)$ as its left-hand endpoint is of degree $3n$; we indicate this absence from $\B^n(h_t)$ by shading it.}

\label{action_inception}

\end{figure}

\begin{remark}  The conventions in this paper for the visualization of a degree $n$ barcode are as follows:  red endpoints correspond to degree $n$ actions, blue to degree $n+1$ actions, black to an action of any other degree, and any bars belonging to a barcode of another degree will be lightly shaded.  Furthermore, an endpoint which moves left or right as we move through a family of barcodes will have an arrow next to it.  Action values are measured along the horizontal axis; a bar's height bears no significance.
\end{remark}

We complete our proof via contradiction, and towards this end we let $t_0' \in [0,1]$ to be the infimum of all times $t > t_0$ in $T$ such that $c^{n+1}$ does not appear in $\B^n(h_t)$.  First, note that $t_0' \notin T_0$.  Assuming otherwise, $t_0'$ would be an infimum of times belonging to a set not including $t_0'$, so we would have a sequence of time values $t_+^i \in T$ decreasing towards $t_0'$ for which $c^{n+1}$ does not appear in $B^n(h_{t_+^i})$.  We know by the reasoning above that there exists a non-empty interval of time $(t_{-1}', t_0')$ with $c^{n+1}$ appearing in $\B^n(h_{t})$ for times $t \in (t_{-1}', t_0')$; for such times in $(t_{-1}', t_0')$, define $c^n_t$ to be the degree $n$ action at time $t$ with $[c^n_t, c^{n+1}) \in \B^n(h_{t})$.  We may use the existence of the $t_+^i$ to assert that $c^n_t$ must go to $c^{n+1}$ as $t$ goes to $t_0$.  (Indeed, given any $\eps' > 0$, choose times $t_- \in (\tm1', t_0')$ and $t_+^i$ which give $||h_{t_+^i} - h_{t_-} ||_{L^{\infty}} < \eps'$ so that an $\eps'$-matching $\mu_{\eps'}$ exists between the two degree $n$ barcodes; since no bar with $c^{n+1}$ as a right endpoint exists in $\B^n(h_{t_+^i})$, we must have the bar $[c^n_{t_-}, c^{n+1})$ being of length less than $2 \eps '$.)  With $c^n_t$ getting arbitrarily close to $c^{n+1}$ as $t$ increases to $t_0'$, our assumption on $f$'s actions therefore says that $c^n_t$ is an action coming from the $y$-intercept of $h_t$.  Yet if $t_0' \in T_0$, our analysis of our homotopy says this implies $h_t$ will have a degree $n$ action equal to $c^{n+1}$ coming from a kink for all times past $t_0'$ and hence for $t = 1$.  This contradiction of our assumptions on $f$'s actions allows us to conclude that $t_0' \notin T_0$.  

So $t_0'$ has to be in $T$.  Moreover, arguments similar to the ones given above show that $c^{n+1}$ does not appear in $\B^n(h_{t_0'})$ and that $c^n_t$ as previously defined must still be an action coming from the $y$-intercept (and so of the form $h_t(0) + k \sigma(\lambda)\gamma$) which converges to $c^{n+1}$ as $t$ goes to $t_0'$.  We must therefore have $t_0' \neq 1$ to avoid contradicting our assumptions on $f$'s actions.  So we can find a non-empty interval of time $(t_{-1}', t_1') \subset T$ containing $t_0'$ with $||h_{t_-} - h_{t_+}||_{L^{\infty}} < \eps$ for every $t_-, t_+ \in (t_{-1}', t_1')$, implying the existence of $\eps$-matchings for the various pairs $\B^n(h_{t_-})$, $\B^n(h_{t_+})$.  Moreover, $h_t(0) + k \sigma (\lambda) \gamma$ will continue to increase past $c^{n+1}$ as $t$ increases past $t_0'$ while remaining an action of index $n$.

Our proof of Lemma \ref{lemma-cd-n+1} for the case of $\lambda \neq 0$ is nearly complete.  Let $t_+ \in (t_0', t_1')$ be given.  At this time, we have our degree $n$ action $h_{t_+}(0) + k \sigma (\lambda) \gamma$ not equal to any degree $n+1$ or $n-1$ actions, so it must appear in either $\B^{n-1}(h_{t_+})$ or $\B^n(h_{t_+})$.  Moreover, this action is higher in action but lower in degree than any other actions in $I_{\eps}(c^{n+1})$ ($c^{n+1}$ of degrees $n+1$ and $3n$), so it must pair with something outside of $I_{3\eps}(c^{n+1})$.  But the number of degree $n$ actions in $I_{\eps}(c^{n+1})$ pairing with anything outside of $I_{3\eps}(c^{n+1})$ was zero for any time in $(t_{-1}', t_0')$, so Lemma \ref{lemma-very-useful} says that the same should be true for $t_+$.  We therefore have a contradiction of the definition of $t_0'$; see Figure \ref*{toprime_contradiction}.

\begin{figure}
\begin{tikzpicture}[scale=.4]
		\VertDash{-4}{$c^{n+1}$};
		\ParBarsLight{-4}{2}{black}{+};
		\StatBars{-5}{3}{1}{red}{blue};
		\draw[-triangle 90] (-4.65, 3) -- +(0.25,0);
		\draw[-] (-10, 0) -- (4, 0);
		%\draw[-] (-6, 4) -- (-5.25, 3.25);
		%\node[] at (-6.25, 4.25) {$c^n_t$};
		\node[] at (-6, 6.2) {$\B^n(h_t)$ for $t \in (t'_{-1}, t'_0)$ };
\end{tikzpicture}
\hspace{1cm}
\begin{tikzpicture}[scale=.4]
		\node[] at (-7, 6.2) {$\B^n(h_t)$ for $t = t_0'$ };
		\VertDash{-4}{$c^{n+1}$};
		\ParBarsLight{-4}{2}{black}{+};
		%\draw[-] (-5, 4) -- (-4.25, 3.25);
		%\node[] at (-5.5, 4.5) {$c^n_t$};
		\node[fill = red, shape = circle, scale = .5] at (-4, 3) {};
		\draw[-triangle 90] (-3.65, 3) -- +(0.25,0);
		\draw[-] (-10, 0) -- (4, 0);
\end{tikzpicture}

\vspace{2cm}

\begin{tikzpicture}[scale=.4]
		\node[] at (-6, 6.2) {$\B^n(h_t)$ for $t \in (t'_0, t'_1)$ };
		\VertDash{-4}{$c^{n+1}$};
		\ParBarsLight{-4}{2}{black}{+};
		\node[fill = blue, shape = circle, scale = .5] at (-4, 3) {};
		\node[fill = red, shape = circle, scale = .5] at (-3.25, 3) {};
		%\draw[-triangle 90] (-2.90, 3) -- +(0.25,0);
		\draw[-] (-10, 0) -- (4, 0);
		%\draw[-] (-3.25, 3.25) -- (-3.25, 4);
		%\node[] at (-3.25, 4.5) {$c^n_t$};
\end{tikzpicture}

\vspace{0.5cm}

\caption{The evolution of $\B^n(h_t)$ for times close to $t_0'$.  The red, degree $n$ action having nothing to pair with for times immediately following $t_0'$ contradicts the existence of $t_0'$.}
\end{figure}
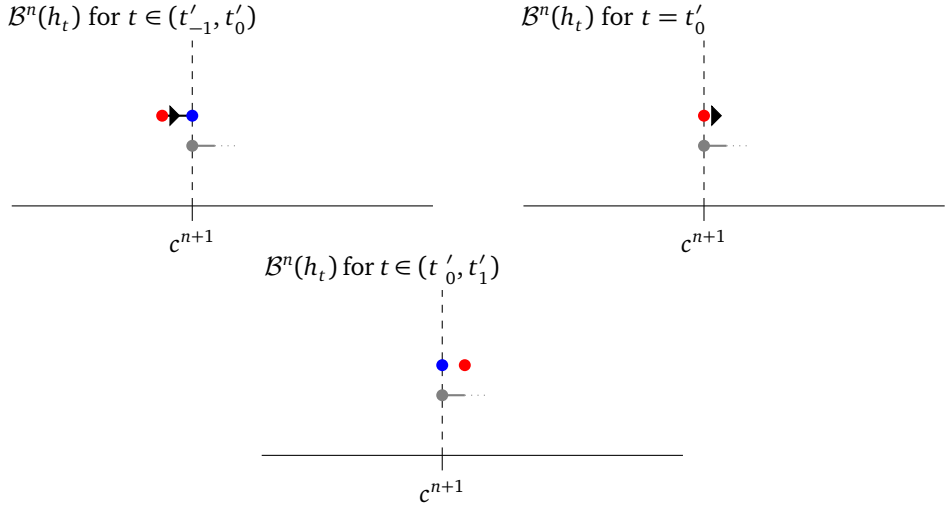

\label{toprime_contradiction}

\vspace{0.5cm}

\textbf{The case that  $\lambda = 0$.}

Before explaining the changes we make for the proof of Lemma \ref{lemma-cd-n+1} in the case that $\lambda = 0$, we take the time to explain their necessity.  There were several instances in the proof of our previous case where we applied Lemma \ref{lemma-very-useful} to compare the barcodes of $h_{t_-}$ with those of $h_{t_+}$ for some times $t_-, t_+$, and then noted that the number of degree $n$ actions in $I_{\eps}(c^{n+1})$ which paired with actions outside of $I_{3 \eps}(c^{n+1})$ was zero for $h_{t_-}$.  Such claims do not always carry over in the case of $\lambda = 0$.

Indeed, for any $\eps > 0$, let $t_0$, $(t_{-1}, t_0)$, $t_0'$, and $(t_{-1}', t_0')$ be as previously defined.  For $t_- \in (t_{-1}, t_0)$, all actions coming from the $y$-intercept will be of the form $h_{t_-}(0)$ (which is within $\eps$ of $c^{n+1} = h_{t_0}(0)$), and if $N = n$, then one such action coming from the $y$-intercept will be of degree $n$.  We therefore cannot say as before that the number of degree $n$ actions in $I_{\eps}(c^{n+1})$ is zero at time $t_-$.  Furthermore, we cannot assert as before that, for times in $(t_{-1}', t_0')$, the number of degree $n$ actions in $I_{\eps}(c^{n+1})$ which pair with something outside of $I_{3\eps}(c^{n+1})$ is zero.  In particular, the case that $N = n$ gives a degree $n$ action precisely equal to $c^{n+1}$ for all times past $t_0$, and this action may very well pair with something outside of $I_{3 \eps}(c^{n+1})$ for times immediately preceding $t_0'$.

With the failures of our previous arguments explained, we now describe how to overcome them.  We have $M$ monotone in the case that $\lambda = 0$, and as described in Section \ref*{HFH}, the Floer differential in this case may be described by a count of solutions to the Hamiltonian Floer equation (though we must choose a generic almost complex structure first).  In this case, if such a solution $u$ connects two capped periodic orbits $\bar{x}$ and $\bar{y}$, then the \textit{energy} of the strip $u$, defined as 
\[
E(u) = \int_{\R} \int_0^1  || \partial_s u ||^2 dt \, ds,
\]
is precisely equal to the difference in the actions of $\bar{x}$ and $\bar{y}$.  We make use of the following lemma, presented here as in \cite{SeyAnnulus}, nearly verbatim.

\begin{lemma}\label{energy-lemma}

Let $V$ denote an open subset of $M$ with (at least)  two distinct smooth boundary components $W_1, W_2$.  Consider a Hamiltonian $H$ which is autonomous in $V$ whose time-1 map $\phi_H^1$ has no fixed points in $V$.  Further assume that $W_1$ and $W_2$ are contained in two distinct level sets of $H$.  Then there exists a constant $\eps(V, H|_V, J|_V) > 0$, depending on the domain $V$ and the restrictions of the Hamiltonian $H$ and the almost complex structure $J$ to the domain $V$, such that if $u: \R \times \R / \Z \rightarrow M$ is a solution to the Hamiltonian Floer equation and intersects $W_1$ and $W_2$, then
\[
E(u) \geq \eps(V, H|_V, J|_V).
\]
\end{lemma}

Recalling the definition of $r_0 = R$ and $r_1$, we choose $V$ to be the subset of $B(2\pi R)$ defined by
\[
V = \left \{ z \in B(2\pi R) \, \, \middle| \, \, r_1 + \frac{R - r_1}{4} \leq \frac{|z|^2}{2} \leq R - \frac{R - r_1}{4} \right \},
\]
and we fix on $M$ a time-dependent $\omega$-compatible almost complex structure $J_0$.

Redefine $T$ in this case to be those values of time for which $h_t$ satisfies our slope condition and equals $m_1(r-R) + f(R)$ on $(r_1, R)$ (this corresponds to having completed the first leg of our homotopy).  We know that $h_t$ is the same on the set $(r_1 + \tfrac{R-r_1}{4}, R - \tfrac{R - r_1}{4})$ for any $t \in T$, so the same may be said for any standard perturbation $\tilde{H}_t$ of such $h_t$ on $V$.

As noted in the introduction, we must pick a regular almost complex structure $J$ for each non-degenerate $\tilde{H}_t$ to have the differential for the Floer chain complex well-defined.  This may lead one to believe that our choices for $J$ may differ on $V$ from perturbation to perturbation.  However (as remarked in \cite{SeyAnnulus}), for every $t \in T$ and any standard perturbation $\tilde{H}_t$ of $h_t$, the periodic orbits of $\tilde{H}_t$ do not enter $V$, implying that our choice of regular almost complex structure may be chosen to equal $J_0$ on $V$.  Hence our $\eps(V, H|_V, J|_V)$ from Lemma \ref{energy-lemma} will work for any standard perturbation $\tilde{H}_t$ of $h_t$ and any $t \in T$.

With Lemma \ref{energy-lemma} introduced, we continue with our proof.  The first part of the proof is essentially the same as the case that $N, \lambda \neq 0$.    We let $t_0$ be the time of $c^{n+1}$'s inception and say that $s(t_0) = 2 \pi l$.  However, we now choose $4\eps$ to be the minimum of the $\eps(V, H|_V, J|_V)$ from Lemma \ref{energy-lemma} and the minimal positive distance between all kink actions of $f$.  With this $\eps$ chosen, we choose appropriate time intervals $(t_{-1}, t_0), (t_0, t_1)$ as before.  We again know that $s(t_+) > s(t_0)$ for any $t_+ \in (t_0, t_1)$, so the possible actions coming from the $y$-intercept will be of degree

\begin{equation} \label{index-1}
-2ln - n + k2N \tag{**}
\end{equation}
and have the form
\[
h_{t_+}(0).
\]
These, along with the actions $h_{t_0}(0)$ of degree either

\begin{equation} \label{index-2}
-2ln + n + k2N \, \, \text{or} \, \, -2ln - n +1 + k2N \tag{***}
\end{equation}
comprise all actions in $I_{\eps}(c^{n+1})$ at time $t_+$.  By our definition of $t_0$, $h_{t_0}(0) = c^{n+1}$.

We claim that:

\begin{claim}
Our $c^{n+1}$ pairs with $h_{\tp}(0)$ of degree $n$.  Hence, $c^{n+1}$ appears in $\B^n(h_t)$ for $t \in (t_0, t_1)$.

\end{claim}

\begin{proof} (of claim)

%Let $\tmi \in (\tm1, t_0)$ and $\tp \in (t_0, \t1)$ as always.  We know that $||h_{\tp} - h_{\tmi} ||_{L^{\infty}} < \eps$, and that any action of $h_{t_{\pm}}$ not in $I_{\eps}(c^{n+1})$ is outside of $I_{3 \eps}(c^{n+1})$.  We can therefore find small enough standard perturbations $\tilde{H}_{-}$ and $\tilde{H}_{+}$ with $\smallint_0^1 ||\tilde{H}_{+} -  \tilde{H}_{-}||_{L^{\infty}} dt < \eps$ (where $t$ in this expression is the $\R / \Z$ parameter and not the homotopy parameter) and such that actions of $\tilde{H}_{\pm}$ not in $I_{\eps}(c^{n+1})$ are outside of $I_{3 \eps}(c^{n+1})$.  With $\tilde{H}_{-}$ not having any degree $n+1$ actions in $I_{\eps}(c^{n+1})$, we can therefore use a variation of lemma \ref{lemma-very-useful} to say that any degree $n+1$ actions of $\tilde{H}_{+}$ in $I_{\eps}(c^{n+1})$ must pair with something in $I_{\eps}(c^{n+1})$.

Let $\tmi \in (\tm1, t_0)$ and $\tp \in (t_0, \t1)$ as always, and let $\eps'$ be small enough so that $I_{\eps'}(c^{n+1}) \cap I_{\eps'}(h_{\tp}(0)) = \emptyset$.  Note that $\eps' < \eps$ by our assumption on $f$'s actions.  Since $||h_{\tp} - h_{\tmi} ||_{L^{\infty}} < \eps$, we can find standard perturbations $\tilde{H}_{-}$ and $\tilde{H}_{+}$ with $\smallint_0^1 ||\tilde{H}_{+} -  \tilde{H}_{-}||_{L^{\infty}} dt < \eps$ (where $t$ in this expression is the $\R / \Z$ parameter and not the homotopy parameter), and we can choose our perturbations $\tilde{H}_{\pm}$ so that every $c^{\pm} \in \textit{Spec}(\tilde{H}_{\pm})$ is within $\eps'$ of its corresponding action in $\textit{Spec}_m(h_{t_{\pm}})$, with the correspondence being clear when one reviews the proof of Claim \ref{degBarCode} and the discussion preceding it.  Our choice of $\eps$ and perturbations further guarantees that actions of $\tilde{H}_{\pm}$ not in $I_{\eps}(c^{n+1})$ are outside of $I_{3 \eps}(c^{n+1})$.  With $\tilde{H}_{-}$ not having any degree $n+1$ actions in $I_{\eps}(c^{n+1})$, we can therefore use a variation of Lemma \ref{lemma-very-useful} to say that any degree $n+1$ actions of $\tilde{H}_{+}$ in $I_{\eps}(c^{n+1})$ must pair with something in $I_{\eps}(c^{n+1})$.

We know that $\tilde{H}_{+}$ has a degree $n+1$ action in $I_{\eps'}(c^{n+1})$ coming from a capped periodic orbit of the form $[x, v \# k_1A]$, with $v$ the capping contained in $B(2\pi R)$ and $[A]$ as chosen in Remark \ref{DefiningA}.  Since this action must pair with an action in $I_{\eps}(c^{n+1})$, and since the Floer differential in the present case is given by a count of solutions to the Hamiltonian Floer equation, there must exist a Floer trajectory $u$ between $[x,v \# k_1A]$ and some other capped periodic orbit $[y,w \# k_2A]$ (again, with $w$ contained in $B(2\pi R)$) whose action lies in $I_{\eps}(c^{n+1})$ (see Theorem 6.2 and the beginning of the proof of Theorem 12.3 from \cite{Uzi} for more on why such a trajectory should exist).  The energy of any such $u$ is less than $2\eps$, which by Lemma \ref{energy-lemma} means it must be contained within our symplectic ball.  Such a $u$ must also satisfy $[u \# w \# k_2A] = [v \# k_1A]$ (or possibly $[w \# k_2A] = [u \# v \# k_1A]$) as elements of $\pi_2(M)\cong \pi_2(M, B(2\pi R))$.  With $u$, $v$, and $w$ all lying in $B(2\pi R)$, we conclude that $k_1 = k_2$.

Note that $k_1$ plugged into the expression on the right in (\ref{index-2}) gives $n+1$, the degree of $c^{n+1}$.  Hence, the capped periodic orbits of $\tilde{H}_{+}$ having $k_1$ copies of $A$ attached and action in $I_{\eps}(c^{n+1})$ must have degree $n$ or $3n$ according to the remaining expressions from (\ref{index-1}) and (\ref{index-2}).  Noting that $n \neq 1$ for the case of $N \neq 0, \lambda = 0$ so that $3n \neq n+2$, we must have $[x, v \# k_1A]$'s action pairing with the degree $n$ action which lies in $I_{\eps'}(h_{\tp}(0))$.  Since such a pairing holds for arbitrarily small $\eps'$ and perturbations of $h_{\tp}$, our claim holds.

\end{proof}

Next, define $t_0'$ (as in the $\lambda \neq 0$ case) as the infimum of all times $t \in T$ for which $c^{n+1}$ is not in $B^n(h_t)$.  Recalling that the degrees of actions coming from the $y$-intercept are of the form
\[
-2l'n - n + k2N,
\]
we may use our reasoning from the $\lambda \neq 0$ case to conclude the following:

\begin{itemize}

\item  $t_0' \notin T_0$, so the various degrees for the $h_t(0)$ use the same integer $l'$ for time values right around $t_0'$.

\item  $h_t(0)$ is increasing for times right around $t_0'$.

\item  One value of $k$ makes $-2l'n - n +  k2N = n$.

\item  $h_{t_0'}(0) = c^{n+1}$.

\end{itemize}
Moreover, we may use another energy argument as in the proof of our previous claim to say that our $k$ value from the third item above must be $k_1$, while yet another such energy argument gives us our contradiction:  For any time $t_+$ sufficiently close to but greater than $t_0'$, any small enough perturbation $\tilde{H}_{+}$ of $h_{\tp}$ must have the action $c$ which corresponds to $c^{n+1}$ pairing with the action of an orbit that has $k_1$ recappings by $A$.  This other action must lie in $I_{\eps}(c^{n+1})$, and the only orbits for $\tilde{H}_{\tp}$ which satisfy all of these properties either have an incompatible degree ($3n$ with $n \neq 1$) or have degree $n$ with action higher than $c$.  This concludes the proof of Lemma \ref{lemma-cd-n+1} in the case that $\lambda = 0$ and thus in general.

\end{proof}

With Lemma \ref{lemma-cd-n+1} in hand, we may now prove Theorem \ref{cd-n+1-solo} with ease.

\begin{proof}[Proof of Theorem \ref{cd-n+1-solo}]

Consider $f$'s kinks $\{ (r_i, f(r_i)) \}_{i \geq 1}$.  By moving these points slightly in the $r$ and $y$ directions and connecting them with straight lines, we may create for any $\eps > 0$ a piecewise linear function $g$ so that

\begin{itemize}

\item  $g$ satisfies our slope condition and has distinct kink actions.

\item  $||f-g ||_{L^{\infty}} < \eps.$

\item  For every degree $d$, a natural bijection $\nu^d$ exists between $\textit{Spec}^d_m(f)$ and $\textit{Spec}^d_m(g)$ with $|\nu^d(c) - c| < \eps$ for all $c \in \textit{Spec}^d_m(f)$.
\end{itemize}

Let $\{ c^{n+1}_i \}_{i = 1}^m \subset \textit{Spec}_m^{n+1}(f)$ be the set of index $n+1$ actions of $f$ which are equal to $c^{n+1}$ (so $c^{n+1}$ has multiplicity $m$ in $\textit{Spec}_m^{n+1}(f)$).  By Lemma \ref{lemma-cd-n+1}, we know that every $\nu^{n+1}(c^{n+1}_i)$ must appear in $\B^n(g)$ and hence not in $\B^{n+1}(g)$.  Choose a sequence $\eps_k \rightarrow 0$ and a corresponding sequence of $g_k$ and apply the continuity of the barcode to conclude that $c^{n+1}$ cannot appear in $\B^{n+1}(f)$.

For the second part of the theorem, let $4 \eps$ be the minimal distance between $c^{n+1}$ and any degree $n$ action of $f$, and choose a function $g$ as above corresponding to $\eps$.  Then again, any $\nu^{n+1}(c^{n+1}_i)$ appears in $\B^n(g)$, and by our choice of $\eps$, it will be the endpoint of a bar of length at least $2 \eps$.  Hence, the $\eps$-matching $\mu_{\eps}$ between $\B^n(f)$ and $\B^n(g)$ has this bar in its range, and its preimage must be a bar in $\B^n(f)$ with right endpoint $c^{n+1}$ and length at least $4 \eps$.  
\end{proof}

Finally, making slight alterations to the proofs of Lemma \ref{lemma-cd-n+1} and Theorem \ref{cd-n+1-solo} yields the following, which are just as essential as Theorem \ref{cd-n+1-solo} to proving our main theorem.

\begin{theorem}\label{lemma-cd-n-1}  Let $f$ be any piecewise linear function satisfying our slope condition, and let $c^{-3n+1}$ be an action which in degree $-3n+1$ only comes from concave down kinks of $f$.  Then $c^{-3n+1}$ does not enter into $\B^{-3n+1}(f)$, and if no degree $-3n$ action equals $c^{-3n+1}$, then $c^{-3n+1}$ must appear in $\B^{-3n}(f)$.

\end{theorem}

\begin{theorem}\label{lemma-cu}  Let $f$ be any piecewise linear function satisfying our slope condition, and let $c$ represent an action which, in degree $n+1$ (respectively, $-3n+1$), only comes from concave up kinks in $f$'s graph.  Then $c$ does not enter into $\B^{n}(f)$ (resp. $\B^{-3n}(f)$).  Furthermore, if no degree $n+2$ (resp. $-3n+2$) actions equal $c$, then $c$ appears as the left-hand endpoint of a bar in $\B^{n+1}(f)$ (resp. $\B^{-3n+1}(f)$).

\end{theorem}

\end{section}

\begin{section}{Proof of the main theorem}\label{Main_Proof}

Now suppose we have symplectically embedded our ball $B(2\pi R)$ and let $\eps > 0$.  Separate the interval $(R - \eps, R]$ into the union of intervals
\[
\bigcup_{i=1}^{\infty} [R-\eps + \eps(1/2)^{i}, R-\eps + \eps(1/2)^{i-1}].
\]
To define the functions which will later define our embedding, we start by defining for each $i$ a piecewise linear function $f_i: [0,R] \rightarrow \R$ which is supported in the interval $I_i = [R-\eps + \eps(1/2)^{i}, R-\eps + \eps(1/2)^{i-1}]$.  These $f_i$ are defined as follows:

\begin{itemize}
\item $f_i$ is $0$ at the midpoint $r_{i,2}$ of $I_i$ and on $U_{i,1}$, $U_{i,2}$, where $U_{i,1}$, $U_{i,2}$ are small neighborhoods of $I_i$'s left and right endpoints, respectively.

\item $f_i$ is $2\pi R$ at points $r_{i,1}$, $r_{i,3}$, where the interval $(r_{i,1}, r_{i,3})$ is centered at $r_{i,2}$.

\item $f_i$ is linear and increasing, with slope an irrational multiple of $2 \pi$, from the right-hand endpoint of $U_{i,1}$ to $r_{i,1}$ and from $r_{i,2}$ to $r_{i,3}$.

\item $f_i$ is linear and decreasing, with slope an irrational multiple of $2 \pi$, from $r_{i,1}$ to $r_{i,2}$ and from $r_{i,3}$ to the left-hand endpoint of $U_{i,2}$.

\end{itemize}

See Figure \ref*{sample_two}.

\begin{figure}
\centering

\begin{tikzpicture}[scale=.4]
		\draw[-] (-5.5, 0) -- (5.5, 0);
		\draw[dotted] (-6.5,0) -- (-5.5,0);
		\draw[dotted] (5.5,0) -- (6.5, 0) node[right] {$r$};
		\draw[-, very thick] (-5.5, 0) -- (-5,0);
		\draw[-, very thick] (5,0) -- (5.5, 0);
		\draw[very thick] (-5,0) -- (-4,0) -- (-1,4) -- (0,0) -- (1,4) -- (4,0) -- (5,0);
		\node[fill = black, shape = circle, scale = .5] at (-4,0) {};
		\node[fill = black, shape = circle, scale = .5] at (4,0) {};
	    \node[fill = black, shape = circle, scale = .5] at (-1,4) {};
		\node[fill = black, shape = circle, scale = .5] at (0,0) {};
		\node[fill = black, shape = circle, scale = .5] at (1,4) {};		
		\node[fill = black, shape = circle, scale = .5] at (-5,0) {};
		\node[above right] at (1,4) {$(r_{i,3}, \, 2\pi R)$};
		\node[above left] at (-1,4) {$(r_{i,1}, \, 2\pi R)$};
		\node[below] at (0,0) {$(r_{i,2}, \, 0)$};
		\node[fill = black, shape = circle, scale = .5] at (5,0) {};
\end{tikzpicture}

\caption{One of our $f_i$'s.  The neighborhoods $U_{i,1}$ and $U_{i,2}$ have endpoints marked by the first two and last two nodes, respectively.}

\label{sample_two}

\end{figure}
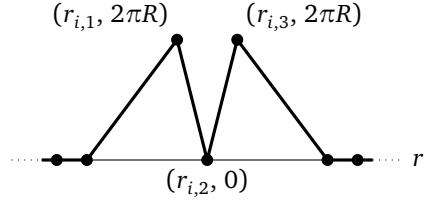

Choose for each $i$ a smooth function $\bar{f}_i$, also supported in $I_i$, which is less than $\eps (1/2)^i$ away from $f_i$ in the $C^0$ norm and has maximum less than $2 \pi R$.  Each such $\bar{f}_i$ induces a Hamiltonian $\bar{F}_i : M \rightarrow \R$, and we define our embedding $ \Phi:  [0,1]^{\infty} \rightarrow Ham(M, \omega)$ by
\[
\Phi(a) = \phi_{\sum_{i=1}^{\infty} a_i \bar{F}_i},
\]
i.e. the sequence $a = \{ a_i \}_{i \geq 1}$ is sent to the Hamiltonian diffeomorphism generated by $\sum_{i=1}^{\infty} a_i \bar{F}_i$.  We will sometimes abuse notation and refer to such diffeomorphisms as being generated by $\sum_{i=1}^{\infty} a_i \bar{f}_i$ instead.

By the definition of the Hofer distance between two Hamiltonian diffeomorphisms, the chain of inequalities from our theorem is equivalent to
\[
2 \pi R||a - b||_{\ell^{\infty}} - \eps \leq ||\Phi(a)^{-1} \circ \Phi(b) ||_H \leq 4 \pi R ||a - b||_{\ell^{\infty}}.
\]
With $\Phi(a)$ being generated by the \textit{autonomous} $\sum_{i=1}^{\infty} a_i \bar{F}_i$, $\Phi(a)^{-1}$ is generated by $\sum_{i=1}^{\infty} -a_i \bar{F}_i$, and since the functions $\sum_{i=1}^{\infty} -a_i \bar{F}_i$ and $\sum_{i=1}^{\infty} b_i \bar{F}_i$ Poisson commute, $\Phi(a)^{-1} \circ \Phi(b)$ is generated by the function $\sum_{i=1}^{\infty} ((b_i - a_i) \bar{F}_i)$.  This expression makes the right-most inequality above trivial.  Indeed, by definition, the Hofer norm of any Hamiltonian diffeomorphism generated by an autonomous function $H$ will be less than or equal to the difference between $H$'s maximum and minimum values, which in turn is less than twice the maximum of its absolute value.  For our function $\sum_{i=1}^{\infty} ((b_i - a_i) \bar{F}_i)$, this quantity is bounded above by $4 \pi R ||a - b||_{\ell^{\infty}}.$

Furthermore, note that $\{ b_i - a_i \}_{ i \geq 1}$ is a sequence with entries in $[-1,1]$,  so that proving the left inequality from our theorem is implied by the following:

\begin{theorem} \label{pen-theorem}

Any function $\sum_{i=1}^{\infty} a_i \bar{f}_i$ with $\{ a_i \} \in [-1,1]^{\infty}$ and $\bar{f}_i$ as above induces a diffeomorphism $\phi$ whose boundary depth $\beta(\phi)$ satisfies $\beta (\phi) \geq 2 \pi R (\text{max}_i|a_i|) - (4 \pi + 7)\eps$.

\end{theorem}

As discussed earlier, this proves the desired inequality since the boundary depth of a Hamiltonian diffeomorphism provides a lower bound for its Hofer norm.  Theorem \ref{pen-theorem} is a direct consequence of the following (to be proven momentarily) and Theorem \ref{theorem-barely}.

\begin{lemma} \label{big-lemma}

Let $a = \{ a_i \}_{i \geq 1}$ be a sequence in $[-1,1]^{\infty}$ with $a_k = \pm 1$ for some $k$, and let $\phi$ be the Hamiltonian diffeomorphism generated by $\sum_{i = 1}^{\infty} a_i \bar{f}_i$.  Then $\beta(\phi) \geq 2 \pi R - (4 \pi + 7 )\eps$.

\end{lemma}

\begin{proof}[Proof of Theorem \ref{pen-theorem}, assuming Lemma \ref{big-lemma}]

Let $\phi$ be generated by $\sum_{i=1}^{\infty}a_i \bar{f}_i$ and let $a_k$ be such that $|a_k| = \text{max}_{i \geq 1}(|a_i|)$.  We may assume that $a_k$ is positive by the following:  For a Hamiltonian diffeomorphism $\phi$, $\beta(\phi) = \beta(\phi^{-1})$ (see \cite{MikeSolo}), and if $\phi$ is generated by autonomous $g$, its inverse is generated by $-g$ as already discussed.  Hence we may replace each $a_i$ with $-a_i$ while leaving the boundary depth unchanged.

Define $b \in [-1,1]^{\infty}$ by setting $b_i = a_i$ for $i \neq k$ and $b_k = 1$, and let $\phi '$ be the Hamiltonian diffeomorphism induced by $\sum_{i = 1}^{\infty} b_i \bar{f}_i$.  By Lemma \ref{big-lemma}, $\beta(\phi ') \geq 2 \pi R - (4 \pi + 7)\eps$.  Let $h_t$ be the straight-line homotopy between the functions $\sum_{i=1}^{\infty} b_i \bar{f}_i$ and $\sum_{i=1}^{\infty} a_i \bar{f}_i$, and let $\phi_t$ be the induced path of diffeomorphisms with $\phi_0 = \phi'$ and $\phi_1 = \phi$.  Theorem \ref{theorem-barely} then tells us that
\[
|\beta(\phi') - \beta(\phi_t)| \leq \underset{[0,R]}{\text{max}} \left( \left( \sum_{i=1}^{\infty}  b_i f_i  \right) - h_t \right) - \underset{[0,R]}{\text{min}} \left( \left( \sum_{i=1}^{\infty} b_i f_i \right) - h_t \right)
\]
(compare the above upper bound to the upper bound from Theorem \ref{theorem-barely}).  

The function $\left( \sum_{i=1}^{\infty} b_i \bar{f}_i \right) - h_t = [t(1 - a_k)]\bar{f}_k$ has maximum less than $[t(1-a_k)]2\pi R$ and minimum zero, so the above inequality becomes
\[
|\beta(\phi') - \beta(\phi_t)| \leq [t(1-a_k)]2\pi R,
\]
leading us to
\[
\beta(\phi') - [t(1-a_k)]2\pi R \leq \beta(\phi_t).
\]
Taking $t = 1$ and using that $\beta(\phi ') \geq 2 \pi R - (4 \pi + 7 )\eps$ finishes the proof.

\end{proof}

The rest of this section is devoted to proving Lemma \ref{big-lemma}.

\begin{proof}

Our goal is to eventually find a bar of the appropriate length in a degree $d$ barcode of some function which is $C^0$-close to $\sum_{i = 1}^{\infty} a_i \bar{f}_i$.  From this, we get a lower bound on our boundary depth and our lemma is proved.  We work out the case where $M$ is monotone with $\lambda > 0$, $n \gamma - 2 \pi N R \geq 0$, and $N \neq 0$ in detail, while a brief discussion of the (slight) modifications necessary for the remaining cases is reserved for the end of the proof.

\vspace{0.5cm}

\textbf{Case 1}  $N \neq 0, \lambda>0,$ and $n \gamma - 2 \pi N R \geq 0$.

Instead of working directly with the function $\sum_{i = 1}^{\infty} a_i \bar{f}_i$, we first pass to its piecewise linear counterpart $\sum_{i = 1}^{\infty} a_i f_i$, after which we will pass to a piecewise linear function $g$ which satisfies our slope condition.

Let $\sum_{i}^{\infty} a_i f_i$ be given with $a$ satisfying our hypothesis.  Assume without loss of generality as in the proof of Theorem \ref{pen-theorem} that $a_k = 1$, and let $r_1 = r_{k,1}$, $r_2 = r_{k,2}$, and $r_3 = r_{k,3}$ be the $r$-values near the center of $f_k$'s support where $f_k$ has kinks (as labeled in Figure \ref*{sample_two}).  We may $C^0$ perturb our graph $\sum_{i}^{\infty} a_i f_i$ by less than $\eps$ to a new piecewise linear function $g$ which has kinks at precisely the same values of $r$ as $\sum_{i}^{\infty} a_i f_i$, satisfies our slope condition, and leaves the points $(r_{\alpha}, f_k(r_{\alpha})), \alpha = 1, 2, 3$ unchanged.  For convenience, we further assume that the the slopes $m_0, m_1$ of the line coming out of the $y$-axis and of the line going into the line $r = R$ are, respectively, negative and positive.  We also assume that $0 < g(R)$.

Define functions $g_0$ and $g_1$ by

\begin{displaymath}
   g_0(r) = \left\{
     \begin{array}{lr}
       m_0 (r-r_2), & 0 \leq r \leq r_2\\
       f_k(r), & r_2 \leq r \leq r_3\\
       m_1 (r - r_3) + 2\pi R , & r_3 \leq r \leq R  
     \end{array}
   \right.
\end{displaymath}

and

\begin{displaymath}
   g_1(r) = \left\{
     \begin{array}{lr}
     m_0(r-r_1) + 2\pi R, & 0 \leq r \leq r_1\\
     f_k(r), & r_1 \leq r \leq r_3\\
       m_1 (r - r_3) + 2\pi R, & r_3 \leq r \leq R  .
     \end{array}
   \right.
\end{displaymath}

The graphs of these functions are displayed in the first and third graphs of Figure \ref*{first_homotopy}.  The function $g_0$ is a $C^0$ approximation of a function which starts off as a constant $0$, then exhibits the rapidly increasing behavior of $f_k$ right after the midpoint of its support, then becomes a constant $2 \pi R$ for the rest of our interval. The function $g_1$ is a $C^0$ approximation of a similar function which exhibits the interesting behavior of $f_k$ on $[r_1, r_3]$ instead.

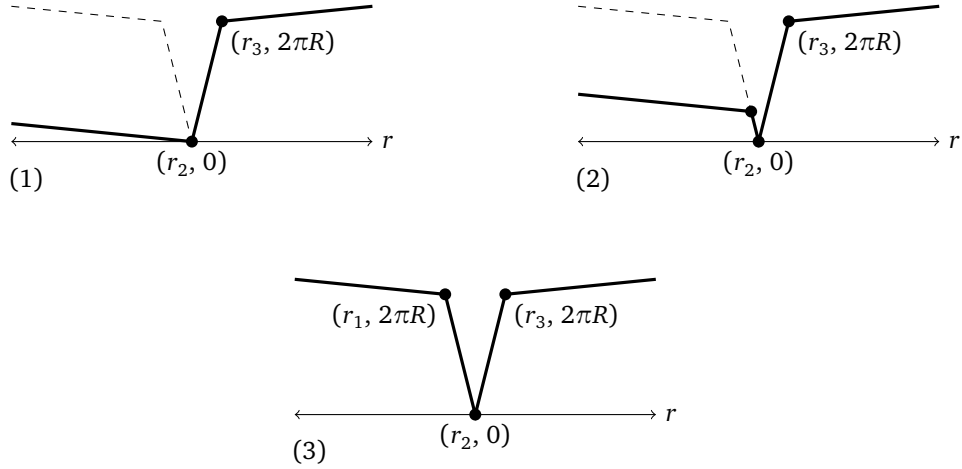
\begin{figure}
\vspace{0.5cm}
\begin{tikzpicture}[scale=.4]
		\node[] at (-5.5, -1.25) {$(1)$};
		\draw[<->] (-6,0) -- (6,0) node[right] {$r$};
		\draw[very thick] (-6,0.6) -- (0,0) -- (1,4) -- (6,4.5);
		\draw[dashed] (-6, 4.5) -- (-1,4) -- (0, 0);
		\node[fill = black, shape = circle, scale = .5] at (1,4) {};
		\node[below right] at (1,4) {$(r_{3}, \, 2\pi R)$};
		\node[fill = black, shape = circle, scale = .5] at (0,0) {};
		\node[below] at (0,0) {$(r_{2}, \, 0)$};
\end{tikzpicture}
\hspace{2cm}
\begin{tikzpicture}[scale=.4]
		\node[] at (-5.5, -1.25) {$(2)$};
		\draw[<->] (-6,0) -- (6,0) node[right] {$r$};
		\draw[very thick] (-6,1.575) -- (-0.25,1) -- (0,0) -- (1,4) -- (6,4.5);
        \draw[dashed] (-6, 4.5) -- (-1,4) -- (-0.25, 1);
        \node[fill = black, shape = circle, scale = .5] at (1,4) {};
		\node[below right] at (1,4) {$(r_{3}, \, 2\pi R)$};
		\node[fill = black, shape = circle, scale = .5] at (0,0) {};
		\node[below] at (0,0) {$(r_{2}, \, 0)$};
		\node[fill = black, shape = circle, scale = .5] at (-0.25, 1) {};
\end{tikzpicture}

\vspace{1cm}
\begin{tikzpicture}[scale=.4]
		\node[] at (-5.5, -1.25) {$(3)$};
		\draw[<->] (-6,0) -- (6,0) node[right] {$r$};
		\draw[very thick] (-6,4.5) -- (-1,4) -- (0,0) -- (1,4) -- (6,4.5);
		\node[fill = black, shape = circle, scale = .5] at (-1,4) {};
		\node[below left] at (-1,4) {$(r_{1}, \, 2\pi R)$};
		\node[fill = black, shape = circle, scale = .5] at (1,4) {};
		\node[below right] at (1,4) {$(r_{3}, \, 2\pi R)$};
		\node[fill = black, shape = circle, scale = .5] at (0,0) {};
		\node[below] at (0,0) {$(r_{2}, \, 0)$};
		
\end{tikzpicture}

\caption{The homotopy $h^1_t$.  The solid graphs in the first and third pictures are $g_0$ and $g_1$, respectively.  The $r$-coordinate of the leftmost kink in the second picture is $r(t)$.}

\label{first_homotopy}

\end{figure}

Let $r(t) = (1-t)r_2+tr_1$.  We connect $g_0$ to $g_1$ via the following homotopy:

\begin{displaymath}
   h^1_t(r) = \left\{
     \begin{array}{lr}
      m_0 (r-r(t)) + 2 \pi R t, & 0 \leq r \leq r(t)\\
     g_1(r), & r(t) \leq r \leq R  .
     \end{array}
   \right.
\end{displaymath}

Notice that the number of kinks in the graph of $h^1_t$ stays the same once the homotopy starts.  Moreover, the slopes around each kink are the same throughout the homotopy, implying that we may parametrize the actions of the $h^1_t$ as functions of time.

First, we give an explicit parametrization of the degree $n+1$ actions which can occur at $r_3$.  Since $h^1_t$ is concave down at $r_3$, we know that the possible degrees occurring here are of the form $-2ln + n + k 2N$ or $-2ln -n +1 + k2N$.  Only the latter of these expressions has values $l$ and $k$ which give it a degree of $n+1$, leading us to focus only on solutions to the equation
\[
-2ln -n + 1 + k2N = n+1
\]
or
\[
2n(-l) + k2N = 2n.
\]
Letting $D$ represent the greatest common divisor of $2n$ and $2N$, we may therefore parameterize our solutions $-l$ and $k$ to the above equation as

\begin{equation} \label{OK-THEN}
-l = 1 - \tfrac{2N}{D}z, \hspace{1cm} k = \tfrac{2n}{D}z \tag{*}
\end{equation}
where $z$ is an integer.  Using our enumeration of actions from Section \ref*{RSH} in the case that $\lambda > 0$, we conclude that any such action has the form
\[
2 \pi \big( 1 - \tfrac{2N}{D}z \big) r_3 + 2 \pi R +  \tfrac{2n}{D} z \gamma;
\]
setting $r_3 = R - \delta_3$ for some $\delta_3 > 0$ and simplifying the above expression yields

\vspace{0.25cm}

\begin{enumerate}
\item[(A1)]  $4 \pi R -2 \pi \delta_3 + \tfrac{2}{D}z(n\gamma - 2 \pi N (R- \delta_3))$
\end{enumerate}

\vspace{0.25cm}

We know that $n\gamma - 2 \pi N(R - \delta_3) > 0$ since $n\gamma - 2 \pi N R \geq 0$, and since $-l < 0$ at $(r_3, f_k(r_3))$, we must have $z > 0$ in (\ref{OK-THEN}) and hence in (A1).  Meanwhile, the inequality $\tfrac{2}{D}(n \gamma - 2 \pi N R) \geq 2 \pi(1 - \tfrac{2N}{D})\delta_3$ implies $\tfrac{2}{D}(n\gamma - 2 \pi N (R - \delta_3)) \geq 2 \pi \delta_3$.  This discussion allows us to conclude that any degree $n+1$ action coming from $(r_3, f_k(r_3))$ is at least as big as $4 \pi R$.

A similar analysis gives that 

\vspace{0.25cm}

\begin{enumerate}
\item[(A2)]  any degree $n$ action coming from $r_3$ will be of the form
\[
2 \pi R + \tfrac{2}{D}z(n\gamma - 2 \pi N r_3),
\]
with $z > 0$.  Hence, all such actions are strictly greater than $2 \pi R$.
\end{enumerate}

\vspace{0.25cm}

Next, we parameterize the relevant actions at $r_2$ and $r(t)$.  Again using our enumeration of actions from Section \ref*{RSH} and calculations similar to those above, we conclude the following.

\begin{enumerate}

\item[(A3)]  Any degree $n+1$ action at time $t$ occurring at $r(t)$ has the form
\[
2 \pi r(t) + 2 \pi Rt + \tfrac{2}{D}z(n \gamma - 2 \pi N r(t)).
\]

\item[(A4)]  Any degree $n$ action at time $t$ occurring at $r(t)$ has the form
\[
2 \pi Rt + \tfrac{2}{D}z(n\gamma - 2 \pi N r(t)). 
\]
We must have $z < 0$.  Using reasoning similar to the case of (A1), we may say that all actions here are no more than $2 \pi r_2$ for all $t \in (0,1]$.

\item[(A5)]  Any degree $n$ action at time $t$ occurring at $r_2$ has the form
\[
2 \pi r_2 + \tfrac{2}{D}z (n \gamma - 2 \pi N r_2).  
\]

\item[(A6)]  We have a degree $n$ action coming from the $y$-intercept of the form $2 \pi Rt - m_0 r(t)$.  By adjusting $m_0$ if necessary, we may assume that $0 < -m_0 r(1) < \eps$.

\item[(A7)]  Any exterior degree $n$ actions will be at least as big as $2 \pi R$.

\item[(A8)]  Any exterior degree $n+1$ actions will be at least as big as $2 \pi R + \gamma \geq 6 \pi R$.

\end{enumerate}

With this new enumeration of actions out of the way, we may continue with our proof.  Consider the degree $n+1$ action $2 \pi r(t) + 2 \pi Rt$ from (A3) with $z=0$, which is easily verified to be a possible value of $z$ if $|m_0|$ was chosen small enough.  This action does not equal any of the previously calculated degree $n+1$ actions for all $t>0$, so if there exists a time $t$ for which another degree $n+1$ action equals our chosen one, it must come from the concave up kink occurring at $r_2$.  There are only finitely many of these.  Similarly, we see that there are only finitely many times where our degree $n+1$ action can equal any degree $n$ action.  Hence, we can find an interval of time right after $t=0$ in which this action is unique among all degree $n+1$ and $n$ actions for $h^1_t$.  Apply Theorem \ref{cd-n+1-solo} to conclude that our degree $n+1$ action must appear in $\B^{n}(h^1_t)$ for this small interval of time.

Furthermore, note that if our degree $n+1$ action limits on a degree $n+1$ action from $r_2$ as $t$ goes to $0$ (implying that said action from $r_2$ has to be $2 \pi r_2$), then we must have $2 \pi (- l) r_2 + k \gamma = 2 \pi r_2$, or equivalently, $k \gamma = 2 \pi (l+1)r_2$.  If $k \neq 0$, we may break this equality by slightly shrinking $\eps$ and thus changing our value of $r_2$.  On the other hand, if $k = 0$, then $l$ would have to be $-1$; again assuming that $|m_0|$ was chosen small enough, this gives an impossible value of $l$ at $r_2$ for $t = 0$.  We are therefore justified in assuming that $2 \pi r_2$ is not a degree $n+1$ action for $h^1_0$.  This and the continuity in $t$ of $\B^{n}(h^1_t)$ imply that our degree $n+1$ action must pair with a degree $n$ action $c^{n}_t$ which is close to it for our previously chosen small interval of time.  In particular, $c^{n}_t$ must satisfy $2 \pi r(t) + 2 \pi Rt - c^{n}_t \rightarrow 0$ as $t \rightarrow 0$, and of the degree $n$ actions enumerated above, the only one to do this is the one with $z=0$ from (A5).

In fact, we claim the following:

\begin{claim}\label{claim-the-first}
Such a pairing persists until such time $\bar{t}$ that the degree $n$ action coming from the $y$-axis, $2\pi R \bar{t} -m_0r(\bar{t})$, equals our chosen degree $n$ action coming from $f_k$'s kink at $r_2$.

\end{claim}

\begin{proof}[Proof of Claim \ref{claim-the-first}]

The idea behind this proof is the following:  Our degree $n$ action will remain stationary for all time, so the only way that the left-hand endpoint of our bar can change is if some degree $n$ action which changes with time eventually equals our chosen one.  However, the only one which can do this occurs at time $\bar{t}$ and is given by the degree $n$ action coming from the $y$-intercept.  On the other hand, our degree $n+1$ action increases with time, so our bar grows until possibly when our degree $n+1$ action equals another; see Figure \ref*{cannot_happen} for a seemingly possible, and troublesome, depiction of how $\B^n(h^1_t)$ changes with time.  But as we shall see, any other degree $n+1$ actions which can equal our chosen one can only come from a concave up kink, which by Theorem \ref{lemma-cu} cannot enter into $\B^{n}(h^1_t)$.  Hence, the scenario depicted in Figure \ref*{cannot_happen} cannot occur.

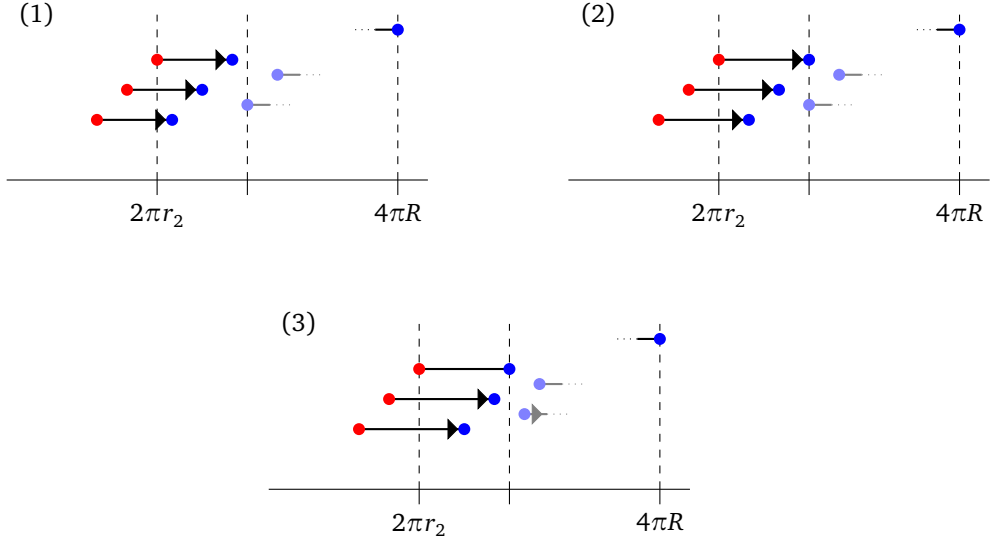
\begin{figure}[ht]
\vspace{0.5cm}
\hspace{3.5cm}
\begin{tikzpicture}[scale=.4]
		\node[] at (-9, 5.5) {$(1)$};
		\VertDash{3}{$4 \pi R$};
		\VertDash{-5}{$2 \pi r_2$};
		\VertDash{-2}{};
		\draw[-] (-10, 0) -- (4, 0);
		\GrowBars{-7}{2}{2.5}{red}{blue};
		\GrowBars{-6}{3}{2.5}{red}{blue};
		\GrowBars{-5}{4}{2.5}{red}{blue};
		\ParBarsLight{-2}{2.5}{blue}{+};
		\ParBarsLight{-1}{3.5}{blue}{+};
		\ParBars{3}{5}{blue}{-};
\end{tikzpicture}
\hspace{6cm}
\begin{tikzpicture}[scale=.4]
		\node[] at (-9, 5.5) {$(2)$};
		\VertDash{3}{$4 \pi R$};
		\VertDash{-5}{$2 \pi r_2$};
		\VertDash{-2}{};
		\draw[-] (-10, 0) -- (4, 0);
		\GrowBars{-7}{2}{3}{red}{blue};
		\GrowBars{-6}{3}{3}{red}{blue};
		\GrowBars{-5}{4}{3}{red}{blue};
		\ParBarsLight{-2}{2.5}{blue}{+};
		\ParBarsLight{-1}{3.5}{blue}{+};
		\ParBars{3}{5}{blue}{-};
\end{tikzpicture}

\vspace{2cm}

\hspace{3cm}
\begin{tikzpicture}[scale=.4]
		\node[] at (-9, 5.5) {$(3)$};
		\VertDash{3}{$4 \pi R$};
		\VertDash{-5}{$2 \pi r_2$};
		\VertDash{-2}{};
		\draw[-] (-10, 0) -- (4, 0);
		\GrowBars{-7}{2}{3.5}{red}{blue};
		\GrowBars{-6}{3}{3.5}{red}{blue};
		\StatBars{-5}{4}{3}{red}{blue};
		\ParBarsLight{-1.5}{2.5}{blue}{+};
		\draw[-triangle 90, black!50] (-1.15, 2.5)-- +(0.25,0);
		\ParBarsLight{-1}{3.5}{blue}{+};
		\ParBars{3}{5}{blue}{-};
\end{tikzpicture}

\vspace{0.5cm}

\caption{A troublesome evolution of $B^n(h^1_t)$, with our bar of choice being the one with left-hand endpoint at $2 \pi r_2$.  In the second picture, our preferred bar has its increasing, degree $n+1$ action switch with a stationary degree $n+1$ action.  This keeps our bar from growing, as depicted in the third picture.  Theorem \ref{lemma-cu}, however, assures us that this cannot happen.}

\label{cannot_happen}

\end{figure}

We will first show that the set
\[
T = \{ t \in (0, \bar{t}) \, | \, [2 \pi r_2, 2 \pi r(t) + 2 \pi Rt) \in \B^n(h^1_t) \}
\]
is open and closed in $(0, \bar{t})$ and so must be equal to $(0, \bar{t})$ since, as we have already seen, $T$ is non-empty.  If $t_0 \in T$ with $t_0 \neq \bar{t}$, we assert the existence of a time interval $(t_{-1}, t_1)$ and an $\eps' >0$ so that for all times $t \in (t_{-1}, t_1)$

\begin{itemize}
\item $(2 \pi r(t_0) + 2 \pi Rt_0) - 2 \pi r_2 > 2 \eps '$.

\item our degree $n+1$ action is at least $\eps'$ away from all other $n+1$ actions for $h^1_{t}$ \textit{except} possibly for constant degree $n+1$ actions of the form $2 \pi r(t_0) + 2 \pi Rt_0$.  In particular, the only other possible degree $n+1$ actions lying in the interval $(2 \pi r(t_0) + 2 \pi Rt_0 - \eps', 2 \pi r(t_0) + 2 \pi Rt_0 + \eps')$ come from the concave up kink of $h^1_{t}$.

\item  our degree $n$ action is at least $\eps'$ away from all other degree $n$ actions of $h^1_{t}$.

\item $||h^1_{t} - h^1_{t_0} ||_{L^{\infty}} < \eps'$.

\end{itemize}

Our third condition may be met since $t_0 \neq \bar{t}$.  So choose $t' \in (t_{-1}, t_1)$.  We know that there should exist an $\eps'$-matching $\mu_{\eps'}$ between $\B^n(h^1_{t_0})$ and $\B^n(h^1_{t'})$.  Our first condition above tells us that $[2 \pi r_2, 2 \pi r(t_0) + 2 \pi Rt_0) \in \B^n(h^1_{t_0})$ is in the domain of $\mu_{\eps'}$.  Furthermore, $\mu_{\eps'}$ should match our degree $n+1$ action $2 \pi r(t_0) + 2 \pi Rt_0$ at time $t_0$ with a degree $n+1$ action which is in $(2 \pi r(t_0) + 2 \pi Rt_0 - \eps', 2 \pi r(t_0) + 2 \pi Rt_0 + \eps ')$, and the only such degree $n+1$ actions at time $t'$ are $2 \pi r(t') + 2 \pi Rt'$ and $2 \pi r(t_0) + 2 \pi Rt_0$.  The latter action, however, can only come from our concave up kink in $h^1_{t'}$'s graph and so cannot enter into $\B^n(h^1_{t'})$ by Theorem \ref{lemma-cu}.  Hence, our degree $n+1$ action $2 \pi r(t_0) + 2 \pi Rt_0$ at time $t_0$ must be matched with the degree $n+1$ action $2 \pi r(t') + 2 \pi Rt'$ at time $t'$.  Similar reasoning shows that our chosen degree $n$ action must be matched with itself between times $t_0$ and $t'$.  Hence, a bar of the form $[2 \pi r_2, 2 \pi r(t') + 2 \pi Rt')$ exists in $\B^n(h^1_{t'})$ for all times $t' \in (t_{-1}, t_1)$, showing that $T$ is an open subset $(0, \bar{t})$.

The set $T$ is closed for a simpler reason:  If $t_0 \in (0, \bar{t})$ is a limit point of $T$, then there are times $t$ immediately prior to (or after) $t_0$ for which a bar of the appropriate form exists in $\B^n(h^1_t)$.  With the lengths of these bars not limiting on something of zero-length as $t$ approaches $t_0$, we may use the continuity of the barcode to say that a bar of the form $[2 \pi r_2, 2 \pi r(t_0) + 2 \pi R t_0)$ must exist in $\B^n(h^1_{t_0})$.  So $t_0$ is in $T$ and consequently $T = (0,\bar{t})$.

To finish the proof of our claim, we note that the same argument used in the previous paragraph shows that $\B^n(h^1_{\bar{t}})$ must have a bar of the appropriate form.

\end{proof}

Similar reasoning shows that for times $t$ bigger than $\bar{t}$, we either have a bar of the form $[2 \pi r_2, 2 \pi r(t) + 2 \pi R t)$ or of the form $[2 \pi Rt - m_0r(t), 2 \pi r(t) + 2 \pi R t)$.  Taking $t =1$, we have a bar of the form $[2 \pi r_2, 2 \pi r_1 + 2 \pi R)$ or of the form $[2 \pi R -m_0r_1, 2 \pi r_1 + 2 \pi R)$ in $\B^n(h^1_1) = \B^n(g_1)$.  Figure \ref*{first_homotopy_end} shows how $\B^n(h^1_t)$ can change with time so that $\B^n(h^1_1)$ will have a bar of the form $[2 \pi R -m_0r_1, 2 \pi r_1 + 2 \pi R)$.

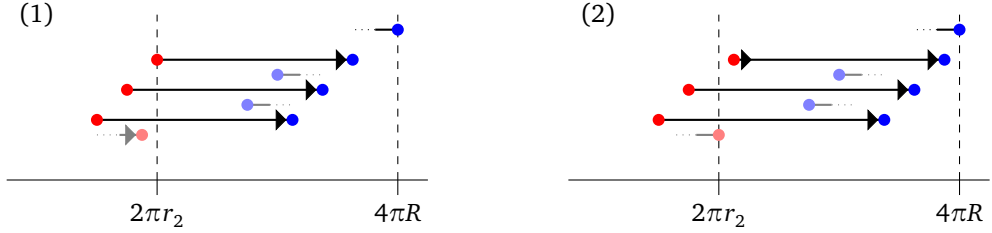
\begin{figure}[ht]
\vspace{0.5cm}
\hspace{3.5cm}
\begin{tikzpicture}[scale=.4]
		\node[] at (-9, 5.5) {$(1)$};
		\VertDash{3}{$4 \pi R$};
		\VertDash{-5}{$2 \pi r_2$};
		\draw[-] (-10, 0) -- (4, 0);
		\GrowParBarsLight{-5.5}{1.5}{red};
		\GrowBars{-7}{2}{6.5}{red}{blue};
		\GrowBars{-6}{3}{6.5}{red}{blue};
		\GrowBars{-5}{4}{6.5}{red}{blue};
		\ParBarsLight{-2}{2.5}{blue}{+};
		\ParBarsLight{-1}{3.5}{blue}{+};
		\ParBars{3}{5}{blue}{-};
\end{tikzpicture}
\hspace{6cm}
\begin{tikzpicture}[scale=.4]
		\node[] at (-9, 5.5) {$(2)$};
		\VertDash{3}{$4 \pi R$};
		\VertDash{-5}{$2 \pi r_2$};
		\draw[-] (-10, 0) -- (4, 0);
		\ParBarsLight{-5}{1.5}{red}{-};
		\GrowBars{-7}{2}{7.5}{red}{blue};
		\GrowBars{-6}{3}{7.5}{red}{blue};
		\GrowBars{-4.5}{4}{7}{red}{blue};
		\draw[-triangle 90] (-4.15, 4) -- +(0.25,0);
		\ParBarsLight{-2}{2.5}{blue}{+};
		\ParBarsLight{-1}{3.5}{blue}{+};
		\ParBars{3}{5}{blue}{-};
\end{tikzpicture}

\vspace{0.5cm}

\caption{How $\B^n(h_t)$ can change between times before and after $\bar{t}$.  In the first picture, the lowest red endpoint corresponds to the degree $n$ action coming from the $y$-axis.  The second picture shows this action eating into our bar after having switched places with the degree $n$ action at $2 \pi r_2$.
%The second picture, which corresponds to time $t = \bar{t}$, has this action switching places with the red, degree $n$ action which corresponds to the left-hand endpoint of our bar of choice.  The third picture depicts the action from the $y$-axis eating into our bar.
}

\label{first_homotopy_end}

\end{figure}

\begin{remark}\label{rmk-exact}  In fact, one may use the homotopy from the proof of Lemma \ref{lemma-cd-n+1} to show that the left-hand endpoint of our bar is \textit{at least} $2 \pi R - m_0r_1$.  In conjunction with the previous paragraph, we conclude that our bar is precisely of the form $[2 \pi R - m_0r_1, 2 \pi r_1 + 2 \pi R)$.
\end{remark}

Next, consider the homotopy $h^2_t$ between $g_1$ and our function $g$ given by

\begin{displaymath}
   h^2_t(r) = \left\{
     \begin{array}{lr}
     \text{max} \{ m_0 (r-r_1) + 2 \pi R(1- 2t), g \}, & 0 \leq r \leq r_1\\
     
     g_1(r), & r_1 \leq r \leq r_3\\
     
     \text{max} \{ m_1 (r-r_3) + 2 \pi R(1- 2t), g\}, & r_3 \leq r \leq R  .
     \end{array}
   \right.
\end{displaymath}

See Figure \ref*{second_homotopy}.

Recall that $2 \pi R - m_0r_1 < 2 \pi R + \eps$.  Our next claim is:

\begin{claim}\label{claim-the-second}
A bar of the form $[c_t, 2 \pi r_1 + 2 \pi R)$, where $c_t < 2 \pi R + \eps$, exists in $\B^n(h^2_{t})$ for all time.

\end{claim}

\begin{proof}[Proof of Claim \ref{claim-the-second}]

The idea behind this proof is that, as can be seen from Figure \ref*{second_homotopy}, the only new degree $n+1$ actions which appear and change with time must be decreasing and either come from concave up kinks in our graph or are exterior actions.  With actions coming only from concave up kinks unable to appear in the degree $n$ barcode, and with any degree $n+1$ exterior actions being bigger than or equal to $4 \pi R$ for all time, our specified right-hand endpoint must be maintained.  On the other hand, our left-hand endpoint can only decrease since all degree $n$ actions which change with time are decreasing during this homotopy.  It is an easy exercise to verify that the homotopy $h^2_t$ satisfies these properties.  Figure \ref*{second_homotopy_barcode} shows how $\B^n(h^2_t)$ changes with time.

\begin{remark}
All degree $n+1$ exterior actions being greater than $4 \pi R$ is due to our assumption that $4 \pi R \leq \gamma$.  Hence, we see here one instance of this assumption's necessity.
\end{remark}

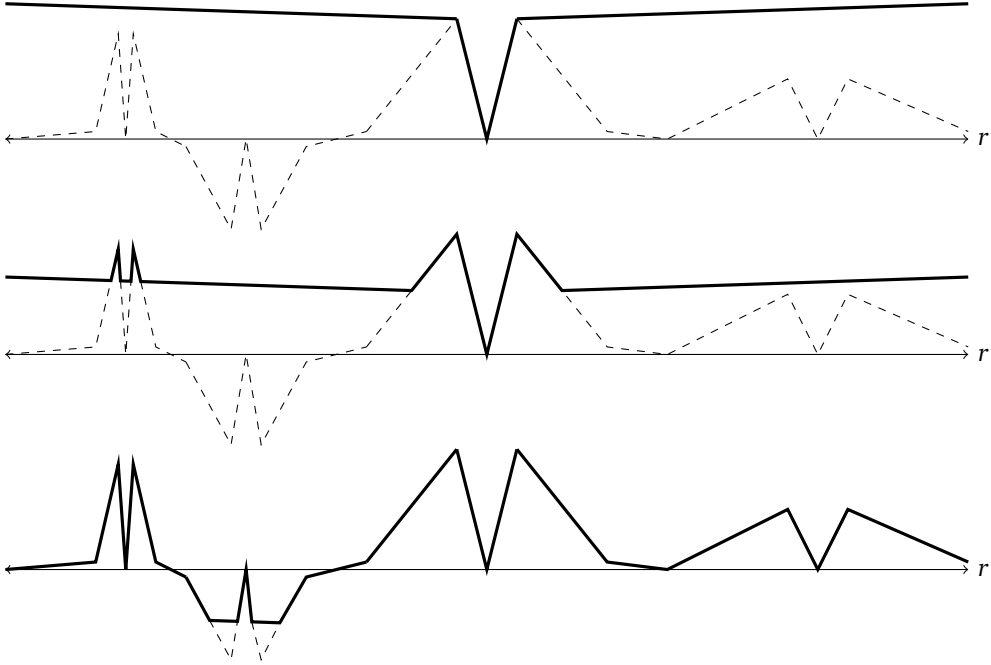
\begin{figure}

\centering

\begin{tikzpicture}[scale=.4]
		\draw[<->] (-16,0) -- (16,0) node[right] {$r$};
		\draw[very thick] (-16, 4.5) -- (-1, 4);
		\draw[very thick] (1, 4) -- (16, 4.5);
		\draw[dashed] (-4,0.25) -- (-1,4);
		\draw[very thick] (-1, 4) -- (0,0) -- (1,4);
		\draw[dashed] (-10,-0.25) -- (-8.5,-3) -- (-8,0) -- (-7.5,-3) -- (-6,-0.25) -- (-4,0.25);
		\draw[dashed] (-16, 0) -- (-13, 0.25) -- (-12.25, 3.5) -- (-12, 0)-- (-11.75, 3.5)-- (-11, 0.25)-- (-10, -0.25);
		\draw[dashed] (1,4) -- (4,0.25) -- (4, 0.25) -- (6, 0) -- (10, 2) -- (11,0) -- (12, 2) -- (16, 0.25);
\end{tikzpicture}

\begin{tikzpicture}[scale=.4]
		\draw[<->] (-16,0) -- (16,0) node[right] {$r$};
		\draw[dashed] (-4,0.25) -- (-1,4);
		\draw[very thick] (-16, 2.575) -- (-12.4905, 2.45802) -- (-12.25, 3.5) -- (-12.1748, 2.44749) -- (-11.826, 2.43587) -- (-11.75, 3.5) -- (-11.5019, 2.42506) -- (-2.5, 2.125) -- (-1, 4) -- (0,0) -- (1,4) -- (2.5, 2.125) -- (16, 2.575);
		\draw[dashed] (-10,-0.25) -- (-8.5,-3) -- (-8,0) -- (-7.5,-3) -- (-6,-0.25) -- (-4,0.25);
		\draw[dashed] (-16, 0) -- (-13, 0.25) -- (-12.4905, 2.45802);
		\draw[dashed] (-12.1748, 2.44749) -- (-12, 0)-- (-11.826, 2.43587);
		\draw[dashed] (-11.5019, 2.42506) -- (-11, 0.25)-- (-10, -0.25);
		\draw[dashed] (1,4) -- (4,0.25) -- (4, 0.25) -- (6, 0) -- (10, 2) -- (11,0) -- (12, 2) -- (16, 0.25);
\end{tikzpicture}

\begin{tikzpicture}[scale=.4]
		\draw[<->] (-16,0) -- (16,0) node[right] {$r$};
		\draw[very thick] (-4,0.25) -- (-1,4);
		\draw[very thick] (-1, 4) -- (0,0) -- (1,4);
		\draw[very thick] (-10,-0.25) -- (-9.21296, -1.6929) -- (-8.2873, -1.7238) -- (-8,0) -- (-7.8125, -1.73958) -- (-6.88036, -1.77232) -- (-6,-0.25) -- (-4,0.25);
		\draw[dashed] (-9.21296, -1.6929) -- (-8.5,-3) -- (-8.2873, -1.7238);
		\draw[dashed] (-7.8125, -1.73958) -- (-7.5, -3) -- (-6.88036, -1.77232);
		\draw[very thick] (-16, 0) -- (-13, 0.25) -- (-12.25, 3.5) -- (-12, 0)-- (-11.75, 3.5)-- (-11, 0.25)-- (-10, -0.25);
		\draw[very thick] (1,4) -- (4,0.25) -- (4, 0.25) -- (6, 0) -- (10, 2) -- (11,0) -- (12, 2) -- (16, 0.25);
\end{tikzpicture}

\caption{The solid graph is $g_1$ in the first picture, while the solid graphs in the second and third pictures are intermediate functions from our homotopy $h^2_t$.}

\label{second_homotopy}

\end{figure}

\vspace{0.25cm}

\begin{figure}[ht]
\vspace{1.75cm}
\hspace{3cm}
\begin{tikzpicture}[scale=.4]
		\node[] at (-9, 8) {$(1)$};
		\draw[dashed] (2.75, 5.75) -- (2.75, 7.5);
		\VertDash{2.75}{$2 \pi r_1 + 2 \pi R$};
		\draw[dashed] (-5, 5.75) -- (-5, 7.5);
		\VertDash{-5}{$2 \pi r_2$};
		\draw[-] (-10, 0) -- (5, 0);
		\ParBarsLight{-5}{1.5}{red}{-};
		\StatBars{-7}{2}{7.75}{red}{blue};
		\StatBars{-6}{3}{7.75}{red}{blue};
		\StatBars{-4.25}{4}{7}{red}{blue};
		\ParBarsLight{-2}{2.5}{blue}{+};
		\ParBarsLight{-1}{3.5}{blue}{+};
		\ParBars{3}{5}{blue}{-};
		\GrowBarsLeft{2.5}{6}{1.5}{blue}{red};
		\GrowBarsLeft{1}{7}{1}{blue}{red};
		\GrowBarsLeftLight{4.25}{5.5}{0.75}{black}{blue};
		\GrowBarsLeftLight{2.75}{4.5}{1.5}{black}{blue};
\end{tikzpicture}
\hspace{6cm}
\begin{tikzpicture}[scale=.4]
		\node[] at (-9, 8) {$(2)$};
		\draw[dashed] (2.75, 5.75) -- (2.75, 7.5);
		\VertDash{2.75}{$2 \pi r_1 + 2 \pi R$};
		\draw[dashed] (-5, 5.75) -- (-5, 7.5);
		\VertDash{-5}{$2 \pi r_2$};
		\draw[-] (-10, 0) -- (5, 0);
		\ParBarsLight{-5}{1.5}{red}{-};
		\StatBars{-7}{2}{7.75}{red}{blue};
		\StatBars{-6}{3}{7.75}{red}{blue};
		\StatBars{-4.25}{4}{7}{red}{blue};
		\ParBarsLight{-2}{2.5}{blue}{+};
		\ParBarsLight{-1}{3.5}{blue}{+};
		\ParBars{3}{5}{blue}{-};
		\GrowBarsLeft{2.5}{6}{3.5}{blue}{red};
		\GrowBarsLeft{1}{7}{3}{blue}{red};
		\GrowBarsLeftLight{4.25}{5.5}{2.75}{black}{blue};
		\GrowBarsLeftLight{2.75}{4.5}{3.5}{black}{blue};
\end{tikzpicture}

\vspace{0.5cm}

\caption{The evolution of $\B^n(h^2_t)$.  Note that the bar on the far right does not have its left-hand endpoint switch with our chosen bar's right-hand endpoint because said left-hand endpoint is an action coming from a concave up kink.}

\label{second_homotopy_barcode}

\end{figure}
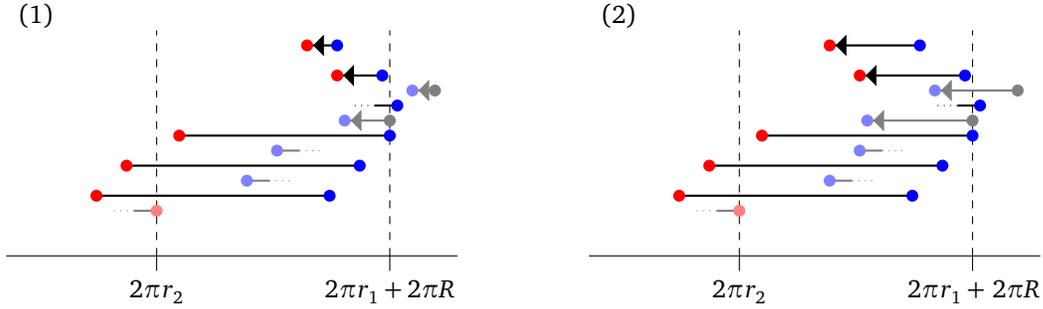

Define $T$ by
\[
T = \{ t \in [0,1] \, | \, [c_t, 2 \pi r_1 + 2 \pi R) \in \B^n(h^2_t) \, , \, c_t \leq 2 \pi r_2 + \eps \}.
\]

Similar to the previous claim's proof, we choose $\eps' > 0$ and an interval of time $(0, t_1)$ so that for any $t \in (0, t_1)$

\begin{itemize}
\item $(0, t_1)$ is small enough so that $||h^2_{0} - h^2_{t}||_{L^{\infty}} < \eps'$.

\item  $(0, t_1)$ is small enough so that no new action values are created.  Hence, we may parameterize all action values as functions of time with domain $(0, t_1)$.

\item  If $c^{n+1}(t)$ is a parameterization of a degree $n+1$ action with domain $(0, t_1)$ which limits to a value in $(2 \pi r_1 + 2 \pi R - \eps', 2 \pi r_1 + 2 \pi R + \eps')$ as $t$ goes to $0$, then said limit is $2 \pi r_1 +2 \pi R$.  Furthermore, if such a $c^{n+1}(t)$ has $c^{n+1}(t') \neq 2 \pi r_1 + 2 \pi R$ for some $t' \in (0, t_1)$, then $c^{n+1}(t)$ is always less than $2 \pi r_1 + 2 \pi R$.

\item  Similarly, if $c^{n}(t)$ is a parameterization of a degree $n$ action with domain $(0, t_1)$ which limits to a value in $(c_{0} - \eps ', c_{0} + \eps ')$ as $t$ goes to $0$, then said limit is $c_{0}$.  Furthermore, if such a $c^{n}(t)$ has $c^{n}(t') \neq c_{0}$ for some $t' \in (0, t_1)$, then $c^{n}(t)$ is always less than $c_{0}$.

%\item  Any degree $n+1$ action for $h^2_{t}$ not lying in the interval $(2 \pi r_1 + 2 \pi R - \eps', 2 \pi r_1 + 2 \pi R + \eps')$ is at least $3\eps'$ away from $2 \pi r_1 + 2 \pi R$.

%\item  Similarly, any degree $n$ action for $h^2_{t}$ not lying in $(c_{0} - \eps', c_{0} + \eps')$ is at least $3\eps'$ away from $c_{0}$.

\end{itemize}

The third and fourth conditions may be met in this case because our actions are non-increasing with respect to time during $(0, t_1)$.  Note our choice of notation $c_t$ (instead of $c(t)$) for the left endpoint of our bar to avoid confusion with the parametrizations of our actions as mentioned in the second condition above.

From here, we prove that $T \neq \emptyset$.  For any $t'$ in $(0, t_1)$, we have an $\eps'$-matching $\mu_{\eps'}$ between $\B^n(h^2_{0})$ and $\B^n(h^2_{t'})$ which must match $2 \pi r_1 + 2 \pi R$ at time $0$ with something in the interval $(2 \pi r_1 + 2 \pi R - \eps', 2 \pi r_1 + 2 \pi R + \eps')$ at time $t'$.  By the conditions above, the only degree $n+1$ actions in $(2 \pi r_1 + 2 \pi R - \eps', 2 \pi r_1 + 2 \pi R + \eps')$ which are not equal to $2 \pi r_1 + 2 \pi R$ at time $t'$ are ones which change with time; such actions in our action window must correspond to concave up kinks in $h_{t'}$'s graph.  Theorem \ref{lemma-cu} therefore states that these do not enter into $\B^n(h^2_{t'})$, and so $2 \pi r_1 + 2 \pi R$ must be matched by $\mu_{\eps'}$ with itself.

Similarly, $\mu_{\eps'}$ must take $c_{0}$ to an action in $(c_{0} - \eps ', c_{0} + \eps ')$, and by construction the only degree $n$ actions in this interval at time $t'$ which are not equal to $c_{0}$ are those strictly less than it.  Hence, any such $\B^n(h^2_{t'})$ has a bar of the desired form, proving $T \neq \emptyset$.

With $T$ non-empty, it has a supremum $t_s \in [0,1]$, and since $T$ is closed (by an argument similar to the one presented at the end of Claim \ref{claim-the-first}'s proof), $t_s \in T$.  Supposing $t_s \neq 1$, the above argument shows that there exists a $t_1 > t_s$ so that $[t_s, t_1) \in T$, contradicting $t_s$ being a supremum.  Hence $t_s = 1$, and our proof is complete.

\end{proof}

Therefore, $\B^n(g)$ has a bar at least as big as $2 \pi r_1 + 2 \pi R - (2 \pi R + \eps) = 2 \pi r_1 -\eps$.  Any standard perturbation $\tilde{G}$ of $g$ which is less than $\eps$ away in the $C^0$ norm will therefore induce a diffeomorphism having boundary depth at least $2 \pi r_1 - 2 \eps$.  Using the notation of Theorem \ref{theorem-barely}, such a standard perturbation will satisfy $|| \sum_{i=1}^{\infty} a_i \bar{f}_i - \tilde{G}|| < 6 \eps$, and so Theorem \ref{theorem-barely} tells us that the Hamiltonian diffeomorphism generated by $\sum_{i=1}^{\infty} a_i \bar{f}_i$ will have boundary depth at least $2 \pi r_1 - 8 \eps$.  Since $r_1$ is strictly greater than $R - \eps$, this lower bound is strictly greater than $2 \pi R - (2 \pi + 8) \eps > 2 \pi R - (4 \pi + 7) \eps$.  This completes the proof of the case that $N \neq 0$, $\lambda >0$, and $n \gamma - 2 \pi N R \geq 0$.

\begin{remark}
%It is now possible for us to see why we may only take coefficients in $[0,1]$ for our quasi-isometry.  Indeed, suppose we take coefficients in some interval with maximum value larger than $1$, so that our function $g_1$ has $\text{max}(g_1) > 2 \pi R$.  Then as explained immediately after the proof of Claim \ref{claim-the-first}, it is possible that our $g_1$ has a bar of the form $[\text{max}(g_1) -m_0r(1), 2 \pi r_1 + \text{max}(g_1))$ instead of one of the form $[2 \pi r_2, 2 \pi r_1 + \text{max}(g_1))$, so that the above reasoning would still give $2 \pi R$ as an approximate lower bound on the boundary depth of $\sum_{i = 1}^{\infty} a_i \bar{f}_i$.

%In fact, in the case that $B(2\pi R)$ is displaceable, we \textit{must} have some impediment to $g_1$ having a bar of the form $[2 \pi r_2, 2 \pi r_1 + \text{max}(g_1))$ with $g_1$ having arbitrarily large maximum.  Otherwise, $\sum_{i = 1}^{\infty} a_i \bar{f}_i$ would have arbitrarily large boundary depth, yet the boundary depth of a Hamiltonian diffeomorphism is bounded above by twice the displacement energy of its support (see \cite{MikeSolo}).

It is now possible to see why our proofs cannot assert Theorem \ref{main-theorem} with $[0,1]^{\infty}$ replaced by $[0, C]^{\infty}$ for some $C > 1$.  Indeed, suppose we tried, so that our function $g_1$ has maximum $2 C \pi R > 2 \pi R$.  Then as explained in Remark \ref{rmk-exact}, our bar of choice in $\B^n(g_1)$ with right-hand endpoint $2 \pi r_1 + 2 C \pi R$ would have left-hand endpoint at least $2 C \pi R - m_0r_1$.  In this best case scenario of $[2 C \pi R - m_0r_1, 2 \pi r_1 + 2 C \pi R)$ being a bar in $\B^n(g_1)$, the reasoning behind Claim \ref{claim-the-second} would still give $2 \pi R$ as an approximate lower bound on the boundary depth of $\sum_{i = 1}^{\infty} a_i \bar{f}_i$.

In fact, in the case that $B(2\pi R)$ is displaceable, we \textit{must} have some impediment to $\sum_{i = 1}^{\infty} a_i \bar{f}_i$ having arbitrarily large boundary depth, for the boundary depth of a Hamiltonian diffeomorphism is bounded above by twice the displacement energy of its support (see \cite{MikeSolo}).
\end{remark}

The rest of this section is devoted to describing the changes necessary to the above proof when dealing with the various other cases.  The only case necessitating any significant changes is the last one, when $N \neq 0$ and $\lambda = 0$.

\vspace{0.5cm}

\textbf{Case 2}  $N = 0$.

The proof given for case 1 can be applied almost directly to the case of $N = 0$, which by our monotonicity assumption implies that $M$ is symplectically aspherical; indeed, the only difference is that our enumeration of $h^1_t$'s actions would exclude those described by (A1), (A2), (A4) and (A8) and only consider the case $z = 0$ for those described by (A3) and (A5).

\vspace{0.5cm}

\textbf{Case 3} $N \neq 0, \lambda > 0,$ and $n\gamma - 2 \pi N R < 0$.

In the case that $N \neq 0$, $\lambda > 0$, and $n \gamma - 2 \pi N <0$, we first pick our $\eps$ to ensure that $n\gamma - 2 \pi N (R - \eps) < 0$ and construct our functions $f_i$.  Again assuming that $a_k = 1$, we choose $g_1$ and $g$ to be as before, but define, $r(t)$, $g_0$, and our first homotopy $h^1_t$ by

\[
r(t) = (1-t)r_2 + tr_3,
\]

\begin{displaymath}
   g_0(r) = \left\{
     \begin{array}{lr}
       m_0 (r-r_1) + 2 \pi R, & 0 \leq r \leq r_1\\
       f_k(r), & r_1 \leq r \leq r_2\\
       m_1 (r - r_2) , & r_2 \leq r \leq R.  
     \end{array}
   \right.
\end{displaymath}

and

\begin{displaymath}
   h^1_t(r) = \left\{
     \begin{array}{lr}
     g_1(r), & 0 \leq r \leq r(t)\\
     m_1(r-r(t))+2\pi Rt, & r(t) \leq r \leq R.\\
     \end{array}
   \right.
\end{displaymath}
Our strategy is to now use the homotopy $h^1_t$ and the continuity of the barcode to establish the existence of a bar either of the form $[-2\pi R + 4 \pi \delta_3, 2 \pi \delta_3)$ or $[-2 \pi r_2, 2 \pi \delta_3)$ in $\B^{-3n}(g_1)$.  We enumerate the relevant degree $-3n$ and $-3n + 1$ actions for $h^1_t$ below, where our notation is as before except that $r(t)$ now represents the right-most kink of $h^1_t$'s graph.

\begin{enumerate}

\item[(B1)]  Degree $-3n+1$ actions from $r_1$ have the form

\[
2 \pi R - 2 \pi r_1 + \tfrac{2}{D}z(n \gamma - 2 \pi N r_1), 
\]

which equals

\[
2 \pi \delta_1 + \tfrac{2}{D}z(n \gamma - 2 \pi N r_1)
\]
if we let $R - \delta_1 = r_1$.  Here, we must have $z<0$.  With $n \gamma - 2 \pi N r_1 < 0$, any such action must be strictly greater than $2 \pi \delta_1$.

\item[(B2)]  Degree $-3n$ actions coming from $r_1$ are of the form

\[
2 \pi R - 4 \pi r_1 + \tfrac{2}{D}z(n \gamma - 2 \pi N r(t)),
\]  

which equals

\[
-2 \pi R + 4 \pi \delta_1 + \tfrac{2}{D}z(n \gamma - 2 \pi N r(t)),
\]

with $z < 0$.  Hence, all such actions are no less than $-2 \pi R + 4 \pi \delta_1$.

\item[(B3)]  Degree $-3n + 1$ actions from $r(t)$ are of the form

\[
2 \pi R t - 2 \pi r(t) + \tfrac{2}{D}z(n \gamma - 2 \pi N r(t)).
\]
We must have $z \geq 0$.  Choosing $z = 0$ here gives the action which will become the right endpoint of our bar.  Note that this action is $2 \pi \delta_3$ at $t = 1$, which is strictly less than any degree $-3n+1$ action from (B1).

\item[(B4)]  A degree $-3n$ action coming from $r(t)$ will have the form

\[
2 \pi Rt - 4 \pi r(t) + \tfrac{2}{D}z(n \gamma - 2 \pi N r(t)).
\]
We must have $z \geq 0$ here (more work than usual must be done to conclude this; see the reasoning following this enumeration).  In particular, any such action is less than the action we get when we take $z = 0$: $2 \pi Rt - 4 \pi r(t)$.  At $t = 1$, this is $-2 \pi R + 4 \pi \delta_3$.  This is the action which may overtake our initial choice of degree $-3n$ action as $t$ gets close to $1$.

\item[(B5)]  A degree $-3n$ action coming from $r_2$ will be of the form

\[
-2 \pi r_2 + \tfrac{2}{D}z(n \gamma - 2 \pi N r_2).
\]
Choose $z = 0$ (giving $-2 \pi r_2 = -2 \pi R + 2 \pi \delta_2$) to get the left endpoint of our bar for times $t$ far enough away from $1$.  Note that this is strictly less than any degree $-3n$ action from (B2).

\item[(B6)]  Any degree $-3n$ action coming from the $y$-axis will be at most

\[
2 \pi R - m_0(r_1) - \gamma \leq -2 \pi R -m_0(r_1)
\]
since we assume that $4 \pi R \leq \gamma$.  Making $|m_0|$ smaller if necessary, we can ensure that this is strictly less than our chosen degree $-3n$ action from (B5).

\item[(B7)]  An exterior degree $-3n$ action will be at most

\[
2 \pi R t + m_1(R - r(t)) - \gamma \leq 2 \pi R t + m_1(R - r(t)) - 4 \pi R
\]
for all times $t$.  Taking $t = 1$ gives $-2 \pi R + m_1(R - r_3) = -2 \pi R + m_1 \delta_3$ as our upper bound, and with $m_1 < 1$, this too is strictly less than our degree $-3n$ action from (B5).

\item[(B8)]  Similarly, any exterior degree $-3n+1$ actions at time $t$ will be at most

\[
2 \pi R t + m_1(R-r(t)) - 4 \pi R.
\]

With $m_1$ small, this is strictly less than our chosen degree $-3n + 1$ action from (B3) for all time.

\end{enumerate}

Our claim concerning the possible values of $z$ for (B3) is due to the following.  In the case of (B3), our parametrization for $l$ (compare with the analysis preceding (A1)) is

\[
-l = -2 - \tfrac{2N}{D}z;
\]
since we are at $r(t)$, we must have $-l < 0$.  This gives that

\[
-2 - \tfrac{2N}{D}z < 0,
\]
and since $\tfrac{2N}{D}$ is a positive integer, we conclude that $z \geq -1$.  Note that we may only include the case $z = -1$ if $2N = D$, and the equality $2N = D$ contradicts the assumptions $n \gamma - 2 \pi N R < 0$ and $4 \pi R \leq \gamma$.  Indeed, the latter assumption yields the first step in the following chain of inequalities:

\begin{align*}
0 & \leq \gamma - 2 \pi R\\
\, & \leq \tfrac{2n}{D}\gamma - 2 \pi R\\
\, & = \tfrac{2n}{D} \gamma - \tfrac{2N}{D}2 \pi R\\
\, & = \tfrac{2}{D}(n \gamma - 2 \pi N R),
\end{align*}
where the second inequality from above uses that $\tfrac{2n}{D}$ is a positive integer.

Choose the degree $-3n+1$ action occurring at $r(t)$ with $z = 0$ (so $2 \pi R t - 2 \pi r(t)$).  Arguing as before, we may assume that $-2 \pi r_2$ is not a degree $-3n+1$ action for $h^1_0$. Use Theorem \ref{lemma-cd-n-1} and the continuity in $t$ of $B^{-3n}(h^1_t)$ to pair our chosen action with the degree $-3n$ action occurring at $r_2$ ($-2 \pi r_2$) for values of $t$ close to $0$, then follow the same reasoning as before to conclude that a bar of the form $[-2 \pi R + 4 \pi \delta_3, 2 \pi \delta_3)$ or $[-2 \pi r_2, 2 \pi \delta_3)$ exists in $\B^{-3n}(g_1)$.  Finally, choose $h^2_t$ as before and follow the reasoning previously given, but employing Theorem \ref{lemma-cu} and the fact that all exterior degree $-3n+1$ actions are strictly less than our chosen one as we perform $h^2_t$, to deduce that $\B^{-3n}(g)$ has a bar of length at least $2 \pi r_2 - 2 \pi \delta_3$.  We may conclude from here that the boundary depth of the Hamiltonian diffeomorphism generated by $\sum_{i = 1}^{\infty} a_i \bar{f}_i$ is at least $2 \pi R - (4 \pi + 7) \eps$.

\vspace{0.5cm}

\textbf{Case 4} $N \neq 0$ and $\lambda < 0$.

For the case of $N \neq 0$ with $\lambda < 0$, we choose our piecewise linear functions and homotopies exactly as in the case of $N \neq 0, \lambda > 0$, and $n \gamma - 2 \pi N R \geq 0$.  We list out the relevant degree $n+1$ and $n$ actions for $h^1_t$ below, from which it should be easy to deduce the appropriate lower bound for the boundary depth.

\begin{enumerate}

\item[(C1)]  Any degree $n+1$ action coming from $r_3$ will be of the form

\[
4 \pi R -2 \pi \delta_3 + \tfrac{2}{D}z(-n\gamma - 2 \pi N (R- \delta_3)).
\]
and we must have $z > 0$.  Since $\tfrac{2}{D}(-n \gamma - 2 \pi N (R - \delta_3))$ is strictly less than $-\gamma \leq - 4 \pi R$, any action described here is negative.

\item[(C2)]  Any degree $n$ action coming from $r_3$ will be of the form

\[
2 \pi R + \tfrac{2}{D}z(-n\gamma - 2 \pi N r_3),
\]
with $z > 0$.  Similar to the case of (B1), any action here is negative.

\item[(C3)]  Any degree $n+1$ coming from $r(t)$ will be of the form

\[
2 \pi r(t) + 2 \pi Rt + \tfrac{2}{D}z(-n \gamma - 2 \pi N r(t))
\]
where we must have $z \leq 0$.  Choose $z = 0$ to get our right-hand endpoint.

\item[(C4)]  Any degree $n$ action coming from $r(t)$ will be of the form

\[
2 \pi Rt + \tfrac{2}{D}z(-n\gamma - 2 \pi N r(t)),
\]
where we must have $z < 0$.  Hence, all actions here are at least as big as $2 \pi Rt + \gamma \geq 2 \pi R t + 4 \pi R$.

\item[(C5)]  Any degree $n$ action coming from $r_2$ will be of the form

\[
2 \pi r_2 + \tfrac{2}{D}z (-n \gamma - 2 \pi N r_2).
\]
Choose $z = 0$ here to get our initial left-hand endpoint.

\item[(C6)]  We have a degree $n$ action coming from the $y$-intercept of the form $2 \pi Rt - m_0 r(t)$.  By adjusting $m_0$ if necessary, we may assume that $0 < -m_0 r(1) < \eps$.  This might eventually overtake our initial choice of degree $n$ action.

\item[(C7)]  We have a degree $n$ exterior action of $2 \pi R + m_1(R-r_3)$, while all others will be no more than $2 \pi R + m_1(R - r_3) - \gamma \leq  -2 \pi R + m_1(R - r_3)$.  Note that these actions never equal our degree $n$ action of choice.

\item[(C8)]  Any degree $n+1$ exterior actions will be no more than $ -2 \pi R + m_1(R - r_3)$, which is strictly less than our degree $n+1$ action of choice for all time.

\end{enumerate}

\vspace{0.5cm}

\textbf{Case 5} $N \neq 0$ and $\lambda = 0$.

Finally, we deal with the case that $\lambda = 0$.  Where $g_1$ is as always, our strategy is to again establish the existence of a bar of the appropriate length in $\B^{-3n}(g_1)$ via the homotopy $h^1_t$ given for the case of $N \neq 0$ and $n \gamma - 2 \pi N R <0$.  The actions are given below.

\begin{enumerate}

\item[(D1)]  Degree $-3n+1$ actions from $r_1$ have the form

\[
2 \pi R - 2 \pi r_1 + \tfrac{2}{D}z(- 2 \pi N r_1), 
\]

which equals

\[
2 \pi \delta_1 + \tfrac{2}{D}z(- 2 \pi N r_1),
\]
if we let $R - \delta_1 = r_1$.  Here, we must have $z<0$.  With $- 2 \pi N r_1 < 0$, any such action must be strictly greater than $2 \pi \delta_1$.

\item[(D2)]  Degree $-3n$ actions coming from $r_1$ are of the form

\[
2 \pi R - 4 \pi r_1 + \tfrac{2}{D}z( - 2 \pi N r(t)),
\]  

which equals

\[
-2 \pi R + 4 \pi \delta_1 + \tfrac{2}{D}z ( - 2 \pi N r(t)),
\]

with $z < 0$.  Hence, all such actions are no less than $-2 \pi R + 4 \pi \delta_1$.

\item[(D3)]  Degree $-3n + 1$ actions from $r(t)$ are of the form

\[
2 \pi R t - 2 \pi r(t) + \tfrac{2}{D}z(- 2 \pi N r(t)).
\]
We must have $z \geq 0$.  Choosing $z = 0$ here gives the action which will become the right endpoint of our bar.

\item[(D4)]  A degree $-3n$ action coming from $r(t)$ will have the form

\[
2 \pi Rt - 4 \pi r(t) + \tfrac{2}{D}z(- 2 \pi N r(t)).
\]
We must have $z \geq -1$ here.

\item[(D5)]  A degree $-3n$ action coming from $r_2$ will be of the form

\[
-2 \pi r_2 + \tfrac{2}{D}z( - 2 \pi N r_2).
\]
Choose $z = 0$ to get the left endpoint of our bar for times $t$ far enough away from $1$.

\item[(D6)]  Any degree $-3n$ action coming from the $y$-axis will be precisely

\[
2 \pi R - m_0(r_1).
\]

\item[(D7)]  An exterior degree $-3n$ action will be equal to

\[
2 \pi R t + m_1(R - r(t)).
\]
for all times $t$.

\item[(D8)]  Similarly, any exterior degree $-3n+1$ actions will be equal to

\[
2 \pi R t + m_1(R - r(t)).
\]

\end{enumerate}

What makes this case slightly more difficult than the others occurs when $N = 1$.  Supposing so, the degree $-3n$ action given by (D4) with $z = -1$ will be equal to our chosen degree $-3n+1$ action for all time, so we may not apply Theorem \ref{lemma-cd-n-1} directly.  However, another energy argument as presented in the proof of Lemma \ref{lemma-cd-n+1} when $\lambda = 0$ shows that we must still have our degree $-3n+1$ action pairing with the usual degree $-3n$ action for times $t$ close to zero.  The rest of the proof for this case matches those of the other cases, though now we must worry about an exterior degree $-3n+1$ action overtaking our chosen degree $-3n+1$ action as we perform $h^2_t$.  But this would give a bar of the form $[-2 \pi R + 4 \pi \delta_3, g(0))$ in $\B^{-3n}(g)$; this bar will be of length at least $2 \pi R - 4 \pi \delta_3$, and so we may say that our Hamiltonian diffeomorphism in question has boundary depth at least $2 \pi R - (4 \pi + 7)\eps$.

\end{proof}

\end{section}

\begin{section}{Adaptation:  Embeddings of $\R \oplus [0,1]^{\infty}$}\label{Leftovers}

Let $(M, \omega)$ and $B(2 \pi R)$ be as in the statement of Theorem \ref{quasimorphism-version}.  A consequence of the discussion given in \cite{EP} shows that, due to the existence of a stable homogeneous Calabi quasi-morphism $\mu$, an autonomous function $F$ supported in $B(2 \pi R)$ will induce a Hamiltonian diffeomorphism $\phi_F$ with Hofer norm at least $\tfrac{1}{\text{Vol}(M)}|\int_{M} F \, \omega^n|$.  Choosing such an $F$ with $\int_M F \omega^n \neq 0$, we may construct the one-parameter family of Hamiltonian diffeomorphisms $\phi_{sF}$ whose Hofer norm grows at least linearly in $s$, i.e. $||\phi_{sF}||_H \geq \tfrac{s}{\text{Vol}(M)} | \int_{M}F \, \omega^n|$.

Setting $\psi_s = \phi_{sF}$ for all $s \in \R$, results from \cite{EP} also show that the path $\{ \psi_s \}$ in $\widetilde{Ham}(M, \omega)$ also has $|| \{  \psi_s \} ||_H \geq  \tfrac{s}{\text{Vol}(M)} |\int_{M}F \, \omega^n|$ when $\widetilde{Ham}(M, \omega)$ admits a stable homogeneous Calabi quasi-morphism $\tilde{\mu}$.  On the other hand, if $\widetilde{Ham}(M, \omega)$ does not admit such a $\tilde{\mu}$, we may still use Proposition 7.1.A from \cite{Polt} to assert that this linear growth still occurs if $M$ has a \textit{stably non-displaceable} Lagrangian $L$ with $supp(F) \cap L = \emptyset$.  With this discussion out of the way, we now begin to prove Theorems \ref{quasimorphism-version} and \ref{universal-analogue}.  

First, suppose $Ham(M, \omega)$ admits a Calabi quasi-morphism $\mu$ and that our symplectic ball $B(2 \pi R)$ is displaceable in $M$.  Pick an $\eps >0$ and construct the functions $\bar{f}_i$, $i \geq 1$, which define our embedding from Theorem \ref{main-theorem}.  Afterwards, define a function $f_0:  [0, R] \rightarrow \R$ satisfying

\begin{itemize}
\item $f_0$ is $0$ at $R-3\eps$ and on the interval $[R-\eps - \delta, R]$ with $\delta$ very small.

\item $f_0$ is $2 \pi R$ at $1 - 2\eps$.

\item $f_0$ is $2 \pi R$ on the interval $[0,R-4 \eps]$.

\item $f_0$ is linear and increasing on $[R - 3\eps, R - 2\eps]$ with slope an irrational multiple of $2\pi$.

\item $f_0$ is linear and decreasing, with slopes irrational multiples of $2\pi$, on $[R - 4\eps, R - 3\eps]$ and $[R - 2\eps, R - \eps - \delta]$.

\end{itemize}

The integral over $M$ of the induced $C^0$ function $F_0: M \rightarrow \R$ is strictly greater than $2 \pi R\cdot \text{Vol}(B_{R- 4\eps})$, where $B_{R - 4 \eps}$ is an abbreviation of $B(2 \pi (R - 4 \eps))$.  Therefore, we may choose a smoothing $\bar{f}_0$ of $f_0$, also supported in $[0, R - \eps]$, so that the induced Hamiltonian $\bar{F}_0$ has its integral satisfying the same inequality.

Now consider $\R \oplus [0,1]^{\infty}$, the set of all sequences $a = \{ a_i \}_{i \geq 0}$ with $a_0 \in \R$, $a_i \in [0,1]$ when $i \geq 1$, and with only finitely many non-zero entries.  Then similar to our definition of $\Phi$, we may define a new map $\overline{\Phi}: \R \oplus [0,1]^{\infty}$ into $Ham(M, \omega)$ by making $\overline{\Phi}(a)$ equal to the Hamiltonian diffeomorphism generated by $\sum_{i=0}^{\infty}a_i \bar{f}_i$.  We now prove Theorem \ref{quasimorphism-version}.  
\begin{proof}[Proof of Theorem \ref{quasimorphism-version}]

Let $\bar{F}_i$, $i \geq 0$ be the Hamiltonians induced by the $\bar{f}_i$, and use $V_{4 \eps}$ to denote the difference in the symplectic volumes of $B(2 \pi R)$ and $B_{R - 4 \eps}$.  Similar to the proof of our main theorem, Theorem \ref{quasimorphism-version} will be established if we show

\[
\frac{2 \pi R \cdot \text{Vol}(B_{R - 4 \eps})}{\text{Vol}(M)} \bigg( \text{max}_{i \geq 0} \{ |a_i| \} \bigg) - \text{max} \left\{ (4 \pi + 7) \eps, \, \frac{2 \pi R \cdot V_{4 \eps}} {\text{Vol}(M)}  \right\}  \leq ||\phi_{\bar{F}}||_H
\]

for $\bar{F} = \sum_{i=0}^{\infty}a_i\bar{F}_i$, where $a = \{ a_i \}_{i \geq 0} \in \R \oplus [-1, 1]^{\infty}$.  This will give us the left-hand inequality of our theorem, while the right-hand inequality is again obvious.  In the case that $|a_0| \neq \text{max}_{i \geq 0} \{ | a_i | \}$, we may use the proofs of Theorem \ref{pen-theorem} and \ref{big-lemma} to again arrive at the inequality

\[
2 \pi R (\text{max}_{i \geq 0} \{ |a_i | \} ) - ( 4 \pi + 7)\eps \leq ||\phi||_H,
\]

from which our desired inequality follows.  On the other hand, suppose $|a_0| = \text{max}_{i \geq 0}|a_i|$.  Then the Hamiltonian $\sum_{i=1}^{\infty}a_i \bar{F}_i$ on $M$ is supported in a region of $M$ with volume $V_{4 \eps}$ and has absolute value bounded above by $2 \pi R$.  From the above discussion, we therefore have

\begin{align*}
||\phi_{\bar{F}}||_H & \geq \frac{1}{\text{Vol}(M)} \bigg( | \smallint \bar{F} \, \omega^n | \bigg) \\
\, & \, \\
\, & \geq \frac{1}{\text{Vol}(M)} \left( | \smallint a_0 \bar{F}_0 \, \omega^n | \, - \, | \smallint \sum_{i=1}^{\infty}a_i\bar{F}_i \, \omega^n | \right) \\
\, & \, \\
\, & \geq \frac{1}{\text{Vol}(M)} \left( |a_0| \cdot \smallint \bar{F}_0 \, \omega^n \, - \, \smallint |\sum_{i=1}^{\infty}a_i\bar{F}_i| \, \omega^n \right) \\
\, & \, \\
\, & \geq \frac{1}{\text{Vol}(M)} \Bigg( |a_0| \cdot 2 \pi R \cdot \text{Vol}(B_{R- 4\eps}) \, - \, 2 \pi R \cdot V_{4 \eps} \Bigg) \\
\, & \, \\
\, & = \frac{2 \pi R \cdot \text{Vol}(B_{R - 4 \eps})}{\text{Vol}(M)} \bigg( |a_0| \bigg) - \frac{2 \pi R \cdot V_{4 \eps}} {\text{Vol}(M)}.
\end{align*}

\end{proof}

As for Theorem \ref{universal-analogue}, take our functions $\bar{F}_i, i \geq 0$ as before, and assume that $(M, \omega)$ and $B(2 \pi R)$ satisfy the appropriate hypotheses.  Then for $a \in \R \oplus [0,1]^{\infty}$, set $\widetilde{\Phi}(a)$ equal to the \textit{path} of Hamiltonian diffeomorphisms generated by the Hamiltonian $\sum_{i = 0}^{\infty} a_i \bar{F}_i$.  With $\widetilde{\Phi}$ so defined, the proof of Theorem \ref{universal-analogue} follows that of Theorem \ref{quasimorphism-version}, thanks in part to the obvious inequality

\[
|| \phi_1 ||_H \leq \widetilde{|| \{ \phi_t \} ||}_H
\]
for an element $\{ \phi_t \}$ of $\widetilde{Ham}(M, \omega)$.

\end{section}

\end{document}